%% file: paper_review_response_v2-arxiv.tex
\documentclass[reqno,a4paper]{amsart}
\usepackage[UKenglish]{babel}

\usepackage{xcolor}

\usepackage{epsf, amsmath, amsthm, amssymb}
\usepackage[foot]{amsaddr}
\usepackage{enumerate}
\usepackage{algorithm,algorithmic,float}
\usepackage{hyperref}
\usepackage[pdftex]{graphicx}

\theoremstyle{theorem}
\newtheorem{theorem}{Theorem}[section]
\newtheorem{lemma}[theorem]{Lemma}

\theoremstyle{definition}
\newtheorem{assumption}[theorem]{Assumption}

\newtheorem{remark}[theorem]{Remark}

\newcommand{\E}{\ensuremath{\mathbb{E}}}  
\newcommand{\F}{\ensuremath{\mathbb{F}}}  
\newcommand{\R}{\ensuremath{\mathbb{R}}}

\newcommand{\Prob}{\ensuremath{\mathbb{P}}}

\ifpdf
\hypersetup{
  pdftitle={Exponential Convergence of Policy Improvement Algorithm},
  pdfauthor={B. Kerimkulov, D. \v{S}i\v{s}ka, and L. Szpruch}
}
\fi

\begin{document}

\title[Exponential Convergence of PIA for Controlled Diffusions]{Exponential Convergence and stability of Howard's Policy Improvement Algorithm for Controlled Diffusions}

\date{\today, }
\author{B. Kerimkulov}
\address{\href{http://www.maxwell.ac.uk/migsaa}{Maxwell Institute Graduate School in Analysis and its Applications}, Edinburgh, UK.}
\email{B.Kerimkulov@sms.ed.ac.uk}
\author{D. \v{S}i\v{s}ka}
\address{\href{https://www.maths.ed.ac.uk}{School of Mathematics, University of Edinburgh} and \href{https://vegaprotocol.io}{Vega Protocol}}
\email{D.Siska@ed.ac.uk}
\author{{\L}. Szpruch}
\address{\href{https://www.maths.ed.ac.uk}{School of Mathematics, University of Edinburgh} and \href{https://www.turing.ac.uk}{Alan Turing Institute}}
\email{L.Szpruch@ed.ac.uk}

\thanks{Supported by the Maxwell Institute Graduate School in Analysis and its Applications, a Centre for Doctoral Training funded by the UK Engineering and Physical Sciences Research Council (grant EP/L016508/01), the Scottish Funding Council, Heriot-Watt University and the University of Edinburgh.}



\maketitle

\begin{abstract}
Optimal control problems are inherently hard to solve 
as the optimization must be performed simultaneously with updating the underlying system.
Starting from an initial guess, Howard's policy improvement algorithm separates the step of updating the trajectory of the dynamical system from the optimization and iterations of this should converge to the optimal control.
In the discrete space-time setting this is often the case and even rates of convergence are known. 
In the continuous space-time setting of controlled diffusion the algorithm consists of solving a linear PDE followed by a maximization problem. This has been shown to converge; in some situations, however no global rate is known.
The first main contribution of this paper is to establish global rate of convergence for the policy improvement algorithm and a variant, called here the gradient iteration algorithm. 
The second main contribution is the proof of stability of the algorithms under perturbations to both the accuracy of the linear PDE solution and the accuracy of the maximization step. 
The proof technique is new in this context as it uses the theory of backward stochastic differential equations.
\end{abstract}

\input{paper_review_response_introduction_v2.tex}

\input{paper_review_response_mainpart_v2.tex}

\end{document}

%% file: paper_review_response_introduction_v2.tex
\section{Introduction}
Stochastic control problems arise naturally in a range of applications in engineering, economics, and finance. 
Apart from very specific cases such as linear-quadratic control in engineering or the Merton 
portfolio optimization task in finance,  stochastic control problems typically have no closed form 
solutions and have to be solved numerically. 
In this paper we consider the policy iteration algorithm and gradient iteration
algorithm; see Algorithms~\ref{alg pia} and~\ref{alg via}. 
These are effectively a linearization method for the inherently nonlinear problem
and play an essential role in numerical solutions of stochastic control problems.

We will consider the continuous space, continuous time problem where 
the controlled system is modeled by an $\mathbb R^d$-valued diffusion process.
Let $W$ be a $d'$-dimensional Wiener martingale on a filtered probability space $(\Omega, \mathcal{F}, (\mathcal{F}_t)_{t\ge 0}, \Prob)$. 
Let us fix a finite time $T\in(0,\infty)$ and consider the controlled SDE
\begin{equation}\label{sde}
dX_s=b^{\alpha}(s,X_s)\,ds+\sigma(s,X_s)\,dW_s\,,\,\, s\in [t,T]\,,\,\,\, X_t = x\,.
\end{equation}
Here $\alpha=(\alpha_s)$ is a control belonging to the space of admissible controls
$\mathcal A$, valued in $A\subseteq\R^m$, and we will write $X^{t,x,\alpha}$ to denote the solution of~\eqref{sde} 
which starts from $x$ at time $t$ while being controlled by $\alpha$.
We shall consider the gain functional in the form
\begin{equation}\label{problem}
J(t,x,\alpha):=\E\left[\int_{t}^{T}f^{\alpha}(s,X_s^{t,x,\alpha})ds+g(X_T^{t,x,\alpha})\right]
\end{equation}
for all $(t,x)\in[0,T]\times\R^d$ and $\alpha\in\mathcal{A}$. 
The value function $v=v(t,x)$ is given for all $t\in[0,T]$ and $x\in \mathbb R^d$ by
\begin{equation}\label{cp}
v(t,x)=\sup_{\alpha\in\mathcal{A}}J(t,x,\alpha).
\end{equation}
We wish to solve the optimization problem, i.e., to find either the value function $v$
or the optimal control $\alpha^*$ which achieves the maximum (or, if the supremum
cannot be reached by $\alpha \in \mathcal A$, then an $\varepsilon$-optimal control $\alpha^\varepsilon \in \mathcal A$ such that $v(t,x) \leq J(t,x,\alpha^\varepsilon) + \varepsilon$). 
It is well known that (see, e.g., Krylov~\cite{krylov controlled}) that under reasonable assumptions the value function satisfies the Bellman PDE:
\begin{align}\label{HJB general}
\begin{split}
\partial_tv + \frac{1}{2}\text{tr}(\sigma \sigma^\top D^2_xv)+ \sup_{a\in A}\left(b^aD_xv+f^a\right)&=0\text{ on }[0,T)\times\R^d,\\
v(T,x)&=g(x)\,\,\text{ on } x\in\R^d\,.
\end{split}
\end{align}
Moreover (again see Krylov~\cite{krylov controlled}), it is sufficient to consider  
Markovian controls, i.e., processes $\alpha_s=a(s,X_s^{t,x,\alpha})$ for some measurable function $a:[0,T]\times\R^d \to A$.
Thus if we have obtained the value function, then we can find the optimal control
(if it exists) as
\[
a^\ast(t,x) = \arg \max_{a\in A}\big( b^a(t,x) (D_x v)(t,x) + f^a(t,x)\big)\,.
\]

It is rarely possible to find a closed form solution to~\eqref{HJB general} 
and so various approximations have to be employed. 
One may, for example, choose to use a finite difference method 
to discretize~\eqref{HJB general} and indeed this has been widely studied; 
see, e.g.,~\cite{dong krylov rate} or~\cite{gyongy siska finite} and references therein.
This results in a high dimensional nonlinear system of equations that still retains
the structure of~\eqref{HJB general}. 
To solve this nonlinear system one may apply the Howard's policy improvement
algorithm. 
The rate of convergence would then follow from results available on discrete 
space-time control problems. 
However, to check that the assumptions required for convergence are satisfied 
is not straightforward and moreover it is dependent on the discretization scheme used.

An alternative approach is to linearize~\eqref{HJB general} and to iterate. 
The classical approach is the Bellman--Howard policy improvement/iteration algorithm. 
The algorithm is initialized with a ``guess'' of the Markovian control. 
Given a Markovian control strategy at step $n$ one solves a linear PDE with
the given control fixed and then one uses the solution to the linear PDE to update
the Markovian control. 
In this paper we will show that this policy improvement algorithm (see Algorithm~\ref{alg pia}) and a variant
which we call the gradient iteration algorithm (see Algorithm~\ref{alg via}) converge, under appropriate assumptions, exponentially fast.

Iterative algorithms for the solution of optimal control problems go back to the work of Bellman~\cite{Bellman,Bellman:1957} where the value iteration algorithms 
for finite space-time problems are developed and their convergences are shown. 
Howard~\cite{howard:dp} proposed the policy improvement algorithm in the context 
of the discrete space-time Markovian decision rocess.
Puterman and Brumelle's paper~\cite{puterman79} was one of the first results on the convergence
properties for the policy iteration for MDP problems.
The abstract function space setting employed in the paper applies to both discrete 
and continuous settings. 
Their main observation is that the policy iteration can be viewed as a type of Newton's method.
Hence similar convergence results to those known for Newton's method follow: in particular, if the initial guess is in a neighborhood of the true solution, then the convergence will be quadratic. 
Puterman~\cite{Puterman:1981} applied this in a setting very similar 
to that of this paper to prove quadratic convergence in the neighborhood of the 
limit. 
Santos and Rust~\cite{Santos2004ConvergencePO} consider the discrete time but continous space and controls setting. 
They extend the results of Puterman and Brumelle~\cite{puterman79} to show global convergence, but without global rate, and 
quadratic local convergence rate of policy iteration and superlinear local convergence under more general conditions. In the case of stochastic control problems with jump-diffusion processes, B\"{a}uerle and Rieder~\cite{Rieder} have proved a  convergence result of the Howard's policy improvement algorithm with the help of martingale techniques.
In the fully discrete space and time setting
Bokanowski, Maroso, and Zidani~\cite{Bokanowski2009} have shown global superlinear convergence, under a monotonicity assumption on the matrices defining the control problem. 
Convergence of policy iteration has been recently proved by Jacka and Mijatovi\'c~\cite{jacka miatovic policy} and Jacka, Mijatovi\'c, and Siraj~\cite{jacka miatovic siraj coupling}.
Further, Maeda and Jacka~\cite{maeda jacka evaluation} have shown quadratic local convergence of the policy iteration algorithm for the time-independent control problem.
The local quadratic convergence is similar to the result of Puterman~\cite{Puterman:1981} but the specific control problem is different 
and moreover they employ a completely different technique based on Schauder estimates for linear PDEs. 

The main contributions of this paper are to establish a global rate of convergence 
and stability
for the policy iteration algorithm and a variant, which we call the gradient iteration algorithm. 
The analysis is carried out using backward stochastic differential equations (BSDEs) and to the best knowledge of the authors this is the first 
time BSDEs have been used to study convergence of the policy iteration algorithm.
The assumptions required for this are effectively Lipschitz dependence in the drift, diffusion, instantaneous payoff, and terminal payoff functions and independence of the diffusion matrix on the control; see~\eqref{sde}.
The stability results show that the policy iteration remains stable even if the linear PDE is solved
only approximately and even if the maximization is step performed approximately.  
Moreover they allow one to devise computationally efficient algorithms as they 
show that in the initial steps it is sufficient to solve the linear PDE with very 
low accuracy, and a highly accurate PDE solver is only required for the final few iterations of the algorithms.

The paper is organized as follows. In Section~\ref{sec main} we introduce all the assumptions and notation used throughout the paper. 
In Sections~\ref{sec conv gia} and~\ref{sec conv pia} 
we state and prove the results concerning convergence of the gradient iteration algorithm and the policy improvement algorithm, respectively.
Section~\ref{sec policy improvement} justifies the name ``policy improvement algorithm'' in that it shows that the value functions increase monotonically 
with iterations and it also shows that the algorithm converges under weaker assumptions than those required for obtaining the rate.
Sections~\ref{sec stability lin PDE} and~\ref{sec stability max} prove 
the stability of the algorithms. In Section~\ref{sec example} we present an example that fits the setting of this paper.
Finally, in Appendix~\ref{sec BSDEs}, we collect several known results 
from the theory of BSDEs that are essential for the proofs.

\begin{algorithm}[H]
\caption{Policy improvement algorithm.}
\label{alg pia}
\begin{algorithmic}
\STATE{Initialization: make a guess of the control $a^0 = a^0(t,x)$.}
\WHILE{difference between $v^{n+1}$ and $v^n$ is large}
\STATE{Given a control $a^n = a^n(t,x)$ solve the {\em linear} PDE
\begin{equation}
\label{eq policy iteration pde}
\begin{split}
\partial_t v^n + \frac{1}{2}\text{tr}(\sigma \sigma^\top D^2_x v^n)+b^{a^n} D_x v^n + f^{a^n} & = 0\text{ on }[0,T)\times\R^d\,,\\
v^n (T,\cdot) & = g\,\,\text{ on } x\in\R^d\,.
\end{split}
\end{equation}
}
\STATE{Update the control
\begin{equation}
\label{eq def anplus1}
a^{n+1}(t,x) = \arg \max_{a\in A}\left[\left(b^a D_x v^n+f^a\right)(t,x)\right]\,.	
\end{equation}
}
\ENDWHILE
\RETURN $v^n, a^{n+1}$.
\end{algorithmic}
\end{algorithm}

\begin{algorithm}
\caption{Gradient iteration algorithm.}
\label{alg via}
\begin{algorithmic}
\STATE{Initialization: make a guess of the value function $v^0 = v^0(t,x)$.
}
\WHILE{difference between $v^n$ and $v^{n-1}$ is large}
\STATE{Given value function $v^{n-1} = v^{n-1}(t,x)$ update the control 
	\begin{equation}
	\label{eq gia update control}
	a^{n}(t,x) = \arg \max_{a\in A}\left[ \left(b^a D_x v^{n-1}+f^a\right)(t,x)\right]\,.
	\end{equation}
}
\STATE{Solve the {\em linear} PDE
\begin{equation}
\label{eq value iteration pde}
\begin{split}
\partial_t v^n + \frac{1}{2}\text{tr}(\sigma \sigma^\top D^2_x v^n)
+ b^{a^n} D_x v^{n-1} + f^{a^n}  & = 0\text{ on }[0,T)\times\R^d\,,\\
v^n (T,\cdot) & = g\,\,\text{ on } x\in\R^d\,.
\end{split}
\end{equation}
}
\ENDWHILE
\RETURN $v^n, a^{n}$.
\end{algorithmic}
\end{algorithm}

We would like to emphasize that Algorithm~\ref{alg pia} and Algorithm~\ref{alg via} are different, although they look rather similar. 
In Algorithm~\ref{alg pia}, $v^n$ is the value function for the Markov control $a^{n}$, since it solves the PDE~\eqref{eq policy iteration pde}.
In Algorithm~\ref{alg via} $v^n$ is not the value function for the Markov control $a^n$. 
This is due to the term $b^{a^n}D_x v^{n-1}$ in the linear PDE~\eqref{eq value iteration pde}. 

%% file: paper_review_response_mainpart_v2.tex
\section{Assumptions and Notation}
\label{sec main}
We fix a finite horizon $T\in (0,\infty)$.
We assume that for some $m\in \mathbb N$ we have $A \subseteq R^m$ such that $0\in A$.
This is the space where the control processes $\alpha$ take values. 
We fix a filtered probability space $(\Omega, \mathcal{F}, \F=(\mathcal{F}_t)_{0\le t\le T}, \Prob)$.
Let  $W=(W_t)_{t\in [0,T]}$ be a $d'$-dimensional Wiener martingale on this space. Moreover, we have the following: 
\begin{enumerate}[(i)]
\item 
For $\gamma > 0$ and a predictable process $\phi$ let us define
\[
\|\phi\|_{\mathbb{H}^2_\gamma} := \left(\E\int_{0}^{T}e^{\gamma s}|\phi_s|^2\,ds\right)^{\frac12}\,.
\]
For $\gamma = 0$ we will write $\|\cdot\|_{\mathbb H^2}$.
We will use $\mathbb{H}^2$ to denote the set of all
predictable processes $\phi$ such that $\|\phi\|_{\mathbb H^2} < \infty$. 
Note that the norm $\|\cdot\|_{\mathbb H^2}$ is equivalent to the norm $\|\cdot\|_{\mathbb H^2_\gamma}$ for any $\gamma \geq 0$.
\item Let $\mathcal{S}^{2}$ be the set of real valued $\mathbb F$-adapted
continuous processes $\phi$ on $[0,T]$ such that 
	\[
	\|\phi\|_{\mathcal{S}^{2}}:=\E\left[\sup_{0\le r\le T}|\phi_r|^2\right]<\infty\,.
	\]
\item For adapted processes $\phi$ such that $\int_0^t |\phi_s|^2\, ds < \infty$ almost surely we will define
\[
(\phi \bullet W)_t := \int_0^t \phi_s \, dW_s\,.
\]
\item For any continuous local martingale $M$ let with $(\langle M \rangle_t)_{t\in [0,T]}$ denote the quadratic variation process and moreover let  
\[
\mathcal E(M)_t := \exp\left(M_t - \frac12 \langle M \rangle_t\right)\,.
\]
\end{enumerate}

We are given measurable functions
\[
b:A\times [0,T]\times\R^d\to \R^d\text{ and } \sigma:[0,T]\times\R^d\rightarrow\R^{d\times d'}\,.
\]
The state of the system is governed by the controlled SDE~\eqref{sde}.

\begin{assumption}
\label{assumption controlled SDE}
The functions $b$ and $\sigma$ are continuous in $t$.
There exists $K\ge 0$ 
and such that $\forall x,y\in\R^d,\forall a\in A,\forall t\in[0,T]$,
\begin{equation}\label{lipchitz}
|b^a(t,x)-b^a(t,y)|+|\sigma(t,x)-\sigma(t,y)|\le K|x-y|
\end{equation}
and
\begin{equation}\label{linear growth}
|\sigma(t,x)|\le K(1+|x|)\,,\,\,\,|b^a(t,x)|\le K(1+|x|+|a|)\,.
\end{equation}
\end{assumption}
\noindent Under Assumption~\ref{assumption controlled SDE} we know that for any $(t,x)\in[0,T]\times \mathbb R^d$ 
and for any progressively measurable $A$-valued control process $\alpha=(\alpha_s)$ 
there is a unique strong solution to~\eqref{sde} which we denote $(X^{t,x,\alpha}_s)_{s\in[t,T]}$.
Let
\[
f:A\times [0,T]\times\R^d \to \R\text{  and  } g:\R^d\rightarrow\R
\]
be two given measurable functions. Let us assume the following for the running gain function $f$ and the terminal gain function $g$ appearing in~\eqref{problem}.
\begin{assumption}
\label{assumption rewards}
There is a constant $K\ge0$ such that  $\forall x,y\in\R^d,\forall a\in A,\forall t\in[0,T]$
\begin{equation}\label{loc_Lip}
|g(x)-g(y)|+|f^a(t,x)-f^a(t,y)|\le K|x-y|
\end{equation}
and
\begin{equation}\label{bound_f}
|f^a(t,x)|\le K(1+|x|+|a|)\,,\,\,\,
|g(x)|\le K\,.
\end{equation}

\end{assumption}
\noindent Under Assumption~\ref{assumption rewards} the gain functional $J$ given
by~\eqref{problem} and the value function $v$ given by~\eqref{cp} are well defined.
Moreover, the value function $v$ satisfies the Bellman equation (with derivatives existing almost everywhere, see Krylov~\cite[Chapter 4]{krylov controlled}, or in the sense of viscosity solutions, see, e.g., Pham~\cite{pham book} or Fleming and Soner~\cite{fleming soner}) 
\begin{equation}\label{HJB}
\begin{split}
\partial_tv+\frac{1}{2}\text{tr}(\sigma\sigma^\top D^2_xv)+\sup_{a\in A}\left(b^aD_xv+f^a\right)&=0\text{ on }[0,T)\times\R^n,\\
v(T,x)&=g(x)\,\,\text{ on } x\in\R^d\,.
\end{split}
\end{equation}

Let us now state the additional assumptions required for our convergence result.
\begin{assumption}
\label{ass convexity}
Let us define for each fixed $(t,x,z) \in [0,T] \times \mathbb R^d \times \mathbb R^d$  the function
\begin{equation}
\label{control function}
a(t,x,z) := \arg\max_{a\in A} \left(b^a(t,x)\sigma^{-1}(t,x)z + f^a(t,x)\right).
\end{equation}
We assume that the function $a(t,x,z)$ is measurable.
\end{assumption}

If the function $a\mapsto \left(b^a(t,x)\sigma^{-1}(t,x)z + f^a(t,x)\right)$ is convex for each fixed $(t,x,z)$, which is in $ [0,T] \times \mathbb R^d \times \mathbb R^d$, one can immediately see that Assumption \ref{ass convexity} holds. 
More generally, this assumption can be verified using an appropriate measurable selection theorem. 
For example, if $A$ is compact, then~\cite[Proposition D.5]{Hernandez1996} shows that an appropriate measurable selection exists. 
If $A$ is not compact but $f$ is bounded, then~\cite[Proposition D.6]{Hernandez1996} gives the same conclusion (using also that $z=D_x v(t,x)$ and Remark~\ref{rem grad bdd}).

\begin{assumption}
\label{ass extra for pia}
There are constants $K, \theta \geq 0$ such that the following hold:
\begin{enumerate}
\item (On the drift)
For all $t\in[0,T]$, $x\in\R^d$, $a,a'\in A$,
\begin{equation} \label{ass b lip in a}
|b^a(t,x) - b^{a'}(t,x)| \leq \sqrt{\theta}|a-a'| 
\end{equation}
and for all $t\in[0,T]$, $x\in\R^d$, $a\in A$ we have 
\begin{equation} \label{change_of_measure}
|(b^a\sigma^{-1})(t,x)|<K\,. 
\end{equation}
\item (On the control function)
For all $t\in[0,T]$, $x,x',z,z' \in\R^d$, $a,a'\in A$ we have that 
\begin{equation*}
|a(t,x,z)-a(t,x,z')| \leq \sqrt\theta|z-z'|\,,
\end{equation*}
\begin{equation*}
|a(t,x,z)-a(t,x',z)| \leq K|x-x'|\,\,\,\,\text{and}\,\,\,\, |a(t,0,0)| \leq K\,.
\end{equation*}
\item (On the running reward) 
\begin{equation*}
|f^{a}(t,x)-f^{a'}(t,x)| \leq \sqrt{\theta}|a-a'| \,\,\,\, \forall t\in[0,T], \forall x\in\R^d, \forall a,a'\in A\,.
\end{equation*}
\end{enumerate}
\end{assumption}
\begin{remark}
Under Assumptions~\ref{assumption rewards} and~\ref{ass extra for pia} we have that
for all $t\in[0,T]$, $x,z,z'\in \R^d$ the following hold:
\[
|f^{a(t,x,z)}(t,x)-f^{a(t,x,z')}(t,x)|\le \theta |z-z'|
\]
and
\[
|f^{a(t,x,0)}(t,x)| \leq (K+K^2)(1+|x|)\,.
\]
\end{remark}

Under Assumptions~\ref{assumption controlled SDE}, \ref{assumption rewards},
\ref{ass convexity}, and~\ref{ass extra for pia}  there is an optimal control process and this 
fact will be used to prove the main results.
\begin{remark}
\label{remark alpha star}
Due to results of Krylov~\cite{krylov controlled} we know that~\eqref{HJB general} has a unique solution and moreover
the map $[0,T] \times \mathbb R^d \ni (t,x) \mapsto D_x v(t,x) \in \mathbb R^d$ is bounded; 
see~\cite[Chapter 4, section 1, Theorem 1]{krylov controlled}.
Hence, by Assumptions~\ref{ass convexity} and~\ref{ass extra for pia} 
we know that $(t,x) \mapsto a(t,x,\sigma(t,x)D_x v(t,x))$ is jointly 
measurable and Lipschitz in $x$.
Thus, for each $(t,x)\in [0,T] \times \mathbb R^d$, the SDE
\[
dX_s = b^{a(s,X_s,\sigma(s,X_s)D_x v(s,X_s))}(s,X_s)\,ds + \sigma(s,X_s)\,dW_s\,,\,\,\,s\in[t,T]\,, X_t = x
\]
has a unique solution $X^{t,x}$.
Then by the verification theorem, the process $\alpha^\ast_s := a(s,X_s,\sigma(s,X_s)D_x v(s,X_s))$ is the optimal control process for~\eqref{cp}.
\end{remark}

All the proofs will be completed in a new measure $\hat{\mathbb P}$ on $(\Omega, \mathcal F)$
given in the
following lemma.
We will use $\hat{\mathbb E}$ to denote the expectation under the measure $\hat{\mathbb P}$.

\begin{lemma}
\label{lemma P hat}
Let Assumptions~\ref{assumption controlled SDE} and~\ref{assumption rewards}  
together with~\eqref{change_of_measure} hold.
Let $(t,x) \in [0,T]\times \mathbb R^d$.
Let $X = X^{t,x,\alpha^\ast}$ be the solution to the SDE~\eqref{sde} started from $(t,x)$ and controlled by the optimal control process $\alpha^\ast$.
Then $d\hat{\mathbb P} := \mathcal E((b^{\alpha^\ast}\sigma^{-1})(\cdot,X) \bullet W)_T \, d\mathbb P$ 
is a probability measure equivalent to $\mathbb P$ and the process 
\[
\widehat{W}_s := W_s + \int_0^s b^{\alpha^\ast_r}(r,X_r)\sigma^{-1}(r,X_r) \,dr
\]
is a $\hat{\mathbb P}$-Wiener process.
\end{lemma}
\begin{proof}
This is an immediate consequence of~\eqref{change_of_measure} and Girsanov's theorem.	
\end{proof}
\begin{remark}
\label{rem grad bdd}
From Krylov~\cite[Chapter 4, section 1, Theorem 1]{krylov controlled} 
we get that there is a constant $C>0$ such that for all $(t,x)\in [0,T)\times \mathbb R^d$ we have that $|D_x v(t,x)| \leq C$.
\end{remark}

\section{Convergence of gradient iteration algorithm}
\label{sec conv gia}

The following theorem gives the convergence result for Algorithm~\ref{alg via}.

\begin{theorem}\label{main theorem alg 1}
Let Assumptions~\ref{assumption controlled SDE}, \ref{assumption rewards},
\ref{ass convexity}, and~\ref{ass extra for pia} hold. 
Let $v$ be the solution to~\eqref{HJB general} and let $(v^n)_{n\in \mathbb N}$ be the approximation sequence given by Algorithm~\ref{alg via}.
Then there is $q\in (0,1)$ depending only on $ K,\theta,T$ and the initial guess $v^0=v^0(t,x)$ such that for all
$(t,x)\in [0,T]\times \mathbb R^d$ there exists $C=C(t,x)$ such that
\begin{equation*}
\begin{split}
|v(t,x)-v^n(t,x)|^2\le C(t,x) q^n\,.
\end{split}
\end{equation*}
\end{theorem}

The main idea of the proof consists of noticing that Algorithm~\ref{alg via} 
can be seen as an iteration on the level of BSDEs.
Using Lemma~\ref{lem classic bsde iteration} we see that on the level of 
BSDEs this iteration is contractive. 
Finally we need to use known results on the connection between BSDEs and
solutions to the HJB equation.

\begin{proof}[Proof of Theorem~\ref{main theorem alg 1}]

We prove the main result in several steps. 
First, we show how to rewrite the gradient iteration algorithm as
an iteration on the level of BSDEs. 
On the $n$th step of the algorithm we need to solve the linear PDE with Lipschitz continuous coefficients~\eqref{eq value iteration pde}. 
Let $v^n$ be the solution to~\eqref{eq value iteration pde} and recall that
\[
a^{n}(t,x) = \arg \max_{a\in A}\left((b^aD_x v^{n-1})(t,x)+f^a(t,x)\right) = a(t,x,\sigma(t,x) D_x v^{n-1}(t,x))\,.
\]
Since we are working with the linear
PDE with Lipschitz continuous coefficients, we have $v^n$ in 
$C^{1,2}([0,T)\times \mathbb R^d)$. 
Let $X = X^{t,x,\alpha^\ast}$ be the solution to the SDE~\eqref{sde} started from $(t,x)$ and controlled by the optimal control process $\alpha^\ast$; see Remark~\ref{remark alpha star}.
From It\^o's formula we then get that
\begin{equation*}
\begin{split}
dv^n(s,X_s) = & \left[\partial_t v^n(s,X_s) + \frac{1}{2}\text{tr}(\sigma\sigma^\top D^2_xv^n)(s,X_s) + (b^{\alpha^\ast_s} D_x v^n)(s,X_s) \right]\,ds \\ 
& + (D_x v^n \sigma ) (s,X_s) \, dW_s\\
= & \left[(b^{\alpha^\ast_s} D_x v^n)(s,X_s) - (b^{a^{n}}D_xv^{n-1})(s,X_s)-f^{a^{n}}(s,X_s)\right]\,ds\\
& + (D_x v^n \sigma )(s,X_s) \, dW_s	\,.
\end{split}
\end{equation*}
Let 
\[
F_s(z) := b^{a(s,X_s,z)}(s,X_s)\sigma^{-1}(s,X_s) z + f^{a(s,X_s, z)}(s,X_s)
\]
and
\begin{equation}
\label{eq def of Y}
Y^n_t:=v^n(t,X_t)\,,\,\,\,\,Z^n_t:=\sigma(t,X_t)D_xv^n(t,X_t)\,,\,\,\,\xi:=g(X_T).
\end{equation}
Then we may write 
\begin{equation}
\label{eq bsde val iter in P}
Y_t^n=\xi-\int_{t}^{T}\left[(b^{\alpha^\ast_s}\sigma^{-1})(s,X_s) Z^n_s - F_s( Z^{n-1}_s)\right]ds-\int_{t}^{T}Z^n_s\,dW_s\,.	
\end{equation}
Let $\hat{\mathbb P}$ and $\widehat{W}$ be given by Lemma~\ref{lemma P hat}.
Hence~\eqref{eq bsde val iter in P} becomes
\begin{equation}
\label{eq bsde val iter in Phat}
Y_t^n=\xi + \int_{t}^{T} F_s(Z^{n-1}_s)\,ds-\int_{t}^{T}Z^n_s\,d\widehat W_s\,.
\end{equation}
Consider now the following BSDE:
\begin{equation}
\begin{split}\label{eq: BSDE for value function}
Y_{t'}^{t,x}&=\xi + \int_{t'}^{T} b^{a(s, X_s^{t,x},Z_s^{t,x})}(s,X_s^{t,x})\sigma^{-1}(s,X_s^{t,x})Z_s^{t,x}+f^{a(s, X_s^{t,x},Z_s^{t,x})}(s, X_s^{t,x})\,ds\\
&\qquad- \int_{t'}^{T}Z_s^{t,x}\,d\widehat W_s,\,\,\,t'\in[t,T]\,,
\end{split}
\end{equation}
where the superscript means that the forward process started from $(t,x)$.
Hence, we can define 
\begin{equation}
\label{eq def of w}
w(t,x):=Y^{t,x}_t\,\,\text{ and }\,\,\sigma(t,x)D_xw(t,x):=Z^{t,x}_t.	
\end{equation}
Therefore by \eqref{control function} we have
\begin{equation*}
\begin{split}
b^{a(t, x,z)}(t,x)\sigma^{-1}(t,x)z+f^{a(t, x,z)}(t, x)=\max_{a\in A}\left(b^{a}(t,x)\sigma^{-1}(t,x)z+f^{a}(t,x)\right).
\end{split}
\end{equation*}
Thus, by Pham~\cite[Theorem 6.3.3]{pham book}, the function $w=w(t,x)$ solves the HJB equation~\eqref{HJB general}. 
Notice that here is the crucial point where the fact that we use the optimal control $\alpha^*$ plays a role. 
Indeed with other control processes we couldn't claim that $w$ solves the HJB equation.
By uniqueness of the viscosity solution to the HJB equation (see the strong comparison principle from~\cite[Theorem 4.4.5]{pham book}), we can conclude that $w=v$ and therefore $w$ is the value function of our stochastic control problem. Therefore, the BSDE~\eqref{eq: BSDE for value function} is the BSDE corresponding to the value function. Notice that~\eqref{eq: BSDE for value function} is a quadratic BSDE, since in the generator we have a product of two Lipschitz functions which depend on $Z$. The existence of the solution to~\eqref{eq: BSDE for value function} under our assumptions can be obtained by applying Theorem~\ref{thr well-defined qBSDE from Adrien's paper} in the case when the terminal cost is bounded for our stochastic control problem.

Using Remark~\ref{rem grad bdd}, the fact that $\sigma^{-1}(s,X_s) Z_s = D_x v(s,X_s)$, and Assumption~\ref{ass extra for pia}, for all $s\in[t,T]$ we get that
\begin{equation}
\label{eq F lip val it}
\begin{split}
|F_s(Z_s)& - F_s(Z^{n-1}_s)|\\
&\le
|\sigma^{-1}(s,X_s) Z_s||b^{a(s,X_s, Z_s)}(s,X_s)-b^{a(s,X_s,Z^{n-1}_s )}(s,X_s)|\\
& \qquad +|b^{a(s,X_s,Z^{n-1}_s )}(s,X_s)\sigma^{-1}(s,X_s)||Z_s-Z^{n-1}_s|+\theta|Z_s-Z^{n-1}_s|\\
&\le
(C\theta+K+\theta)|Z_s-Z^{n-1}_s|\,.
\end{split}
\end{equation}

Moreover, recalling $\xi=g(X_T)$, we get that $\xi\in L^2(\Omega, \mathcal F_T, \hat{\mathbb P})$ from the higher moment estimates for the solution of the SDE and from the Lipschitz property of $g$. 
Similarly $F_s(0)\in\hat{\mathbb H}^2$ by Assumption~\ref{ass extra for pia}. 
We may thus apply Lemma~\ref{lem classic bsde iteration} and hence, 
due to~\eqref{eq bsde val iter in Phat} and~\eqref{eq F lip val it}, we have $q \in (0,1)$ and $\gamma \geq 0$ such that for all $t\in [0,T]$
\begin{equation}
\label{eq conclusion of contr lem grad}
e^{\gamma t}\hat{\mathbb E} |Y_t - Y^n_t|^2 + \|Z - Z^n\|_{\hat{\mathbb H}^2_\gamma}^2 \leq q \|Z - Z^{n-1}\|_{\hat{\mathbb H}^2_\gamma}^2\,.
\end{equation}
Therefore, from~\eqref{eq def of w} and~\eqref{eq def of Y}, we have $v(t,x)=Y^{t,x}_t$ and $v^n(t,x)=Y^{n,t,x}_t$ and by~\eqref{eq conclusion of contr lem grad} we obtain
\begin{equation*}
\begin{split}
e^{\gamma t}|v(t,x)-v^n(t,x)|^2&\le \hat{\mathbb E}[e^{\gamma t}|Y^{t,x}_t-Y^{n,t,x}_t|^2]+\|Z^{t,x}-Z^{n,t,x}\|^2_{\hat{\mathbb H}^2_{\gamma}}\\
&\le q^n \|Z^{t,x}-Z^{0,t,x}\|^2_{\hat{\mathbb H}^2_{\gamma}}\,.
\end{split}
\end{equation*} 
Hence
\begin{equation*}
\begin{split}
&|v(t,x)-v^n(t,x)|^2\\
&\qquad\le q^n\hat{\mathbb E}\int_t^T e^{\gamma(T-t)}|\sigma(s,X^{t,x,\alpha^\star}_s)|^2|D_xv(s,X^{t,x,\alpha^\star}_s)-D_xv^0(s,X^{t,x,\alpha^\star}_s)|^2\,ds\,. 
\end{split}
\end{equation*}
This finishes the proof.
\end{proof}

\section{Convergence of policy improvement}
\label{sec conv pia}

\begin{theorem}\label{main theorem alg 2}
Let Assumptions~\ref{assumption controlled SDE}, \ref{assumption rewards},
\ref{ass convexity}, and~\ref{ass extra for pia} hold. 
Let $v$ be the solution to~\eqref{HJB general} and let $(v^n)_{n\in \mathbb N}$ be the approximation sequence given by Algorithm~\ref{alg pia}.
Then there is $q\in (0,1)$ depending only on $ K,\theta,T$ and the initial guess $v^0=v^0(t,x)$ such that for all
$(t,x)\in [0,T]\times \mathbb R^d$ there exists $C=C(t,x)$ such that
\begin{equation*}
\begin{split}
|v(t,x)-v^n(t,x)|^2\le C(t,x) q^n\,.
\end{split}
\end{equation*}
\end{theorem}

The proof of Theorem~\ref{main theorem alg 2} is similar to that of
Theorem~\ref{main theorem alg 1} except that the iteration on the level of
BSDEs is nonstandard.

\begin{proof}[Proof of Theorem~\ref{main theorem alg 2}]
Let $v^n$ be the solution to~\eqref{eq policy iteration pde} and recall that
\[
a^{n}(t,x) = \arg \max_{a\in A}\left((b^a D_x v^{n-1})(t,x)+f^a(t,x)\right) = a(t,x,\sigma(t,x)D_x v^{n-1}(t,x)) \,.
\]
As before, let $X = X^{t,x,\alpha^\ast}$ be the solution to the SDE~\eqref{sde} started from $(t,x)$ and controlled by the optimal control process $\alpha^\ast$; see Remark~\ref{remark alpha star}.
By It\^o's formula
\begin{equation*}
\begin{split}
dv^n(s,X_s) = & \left[\partial_t v^n(s,X_s) + \frac12\text{tr}(\sigma\sigma^\top D^2_x v^n)(s,X_s) + (b^{\alpha^\ast_s} D_x v^n)(s,X_s) \right]\,ds \\
& + (D_x v^n\sigma) (s,X_s) \, dW_s\\
= & \big[(b^{\alpha^\ast_s} D_x v^n)(s,X_s)  - (b^{a^n}D_xv^n)(s,X_s)
- f^{a^n}(s, X_s) \big]\,ds\\
& + (D_x v^n\sigma)(s,X_s) \, dW_s	\,.
\end{split}
\end{equation*}
Let 
\[
F_s(z,Z):=b^{a(s,X_s,z)}(s,X_s)\sigma^{-1}(s,X_s)Z+f^{a(s,X_s,z)}(s,X_s)\,.
\]
Recalling that the control $\alpha^\ast$ and the associated diffusion $X$ are fixed
we can write
\begin{equation}
\label{eq pia bsde iter in P2}
Y_t^n=\xi-\int_{t}^{T}\left[(b^{\alpha^\ast_s}\sigma^{-1})(s,X_s)Z^n_s-F_s(Z_s^{n-1},Z^n_s)\right]ds-\int_{t}^{T}Z^n_sdW_s.
\end{equation}
Let $\hat{\mathbb P}$ and $\widehat{W}$ be given by Lemma~\ref{lemma P hat}.
Then~\eqref{eq pia bsde iter in P2} becomes
\begin{equation}
\label{eq pia bsde iter in P hat2}
Y_t^n=\xi+\int_{t}^{T}F_s(Z_s^{n-1},Z^n_s)ds-\int_{t}^{T}Z^n_sd\widehat{W}_s\,.
\end{equation} 
Similarly as in Theorem~\ref{main theorem alg 1}, consider the BSDE
\begin{equation}
\begin{split}\label{eq: BSDE for value function policy iteration}
Y_{t'}^{t,x}&=\xi + \int_{t'}^{T} b^{a(s, X_s^{t,x},Z_s^{t,x})}(s,X_s^{t,x})\sigma^{-1}(s,X_s^{t,x})Z_s^{t,x}+f^{a(s, X_s^{t,x},Z_s^{t,x})}(s, X_s^{t,x})\,ds\\
&\qquad- \int_{t'}^{T}Z_s^{t,x}\,d\widehat W_s,\,\,\,t'\in[t,T]\,.
\end{split}
\end{equation}
In same way we can show that $v(t,x)=Y^{t,x}_t$ is the value function of our stochastic control problem.
As before, from Krylov~\cite[Chapter 4, section 1, Theorem 1]{krylov controlled}, 
we get that there is a constant $C>0$ such that for all $(t,x)\in [0,T)\times \mathbb R^d$ we have that $|D_x v(t,x)| \leq C$.
Moreover, as before, using Remark~\ref{rem grad bdd}, the fact that $\sigma^{-1}(s,X_s) Z_s = D_x v(s,X_s)$, and Assumption~\ref{ass extra for pia}, for all $s\in[t,T]$ we get that
\begin{equation}\label{lipschitz F policy iteration}
\begin{split}
|F_s(Z_s,Z_s) & - F_s(Z^{n-1}_s,Z^{n}_s)|  \\
& \le\theta|\sigma^{-1}(s,X_s)Z_s||Z_s-Z^{n-1}_s|\\
&\qquad +|b^{a(s,X_s,Z^{n-1}_s)}(s,X_s)\sigma^{-1}(s,X_s)||Z_s-Z^n_s|+\theta|Z_s-Z^{n-1}_s|\\
& \le\theta C|Z_s-Z^{n-1}_s| + K|Z_s-Z^{n}_s|+\theta|Z_s-Z^{n-1}_s|\,.
\end{split}
\end{equation} 
Finally we note that $\xi \in L^2(\Omega, \mathcal F_T, \hat{\mathbb P})$ and $F_s(0,0)\in\hat{\mathbb{H}}^2$,
so by Lemma~\ref{policy convergence estimate}, together with~\eqref{eq pia bsde iter in P hat2} and~\eqref{eq: BSDE for value function policy iteration},
we have $q \in (0,1)$ and $\gamma \geq 0$ such that for all $t\in [0,T]$
\begin{equation}
\label{eq conclusion of contr lem}
e^{\gamma t}\hat{\mathbb E} |Y_t - Y^n_t|^2 + \|Z - Z^n\|_{\hat{\mathbb H}^2_\gamma}^2 \leq q \|Z - Z^{n-1}\|_{\hat{\mathbb H}^2_\gamma}^2\,.
\end{equation}
Similarly as before, using~\eqref{eq conclusion of contr lem}, we conclude that
\begin{equation*}
\begin{split}
&|v(t,x)-v^n(t,x)|^2 \\
&\qquad \leq q^n \hat{\mathbb E} \int_t^T  e^{\gamma (T-t)}|\sigma(s,X^{t,x,\alpha^\ast}_s)|^2|D_x v(s,X^{t,x,\alpha^\ast}_s) - D_x v^0(s,X^{t,x,\alpha^\ast}_s)|^2\,ds\,.
\end{split}
\end{equation*}
This concludes the proof of the theorem.
\end{proof}

\begin{remark}
Consider briefly the situation where the diffusion coefficient 
also depends on the control, i.e., $\sigma = \sigma^a(t,x)$.
%
After applying It\^o's formula to $v^n(s,X_s)$ and substituting the solution to the linear PDE for $v^n$ we get
	\begin{equation*}
	\begin{split}
	dv^n(s,X_s) 
	= & \big[(b^{\alpha^\ast_s} D_x v^n)(s,X_s)+\frac12\text{tr}(\sigma^{\alpha^\ast_s}(\sigma^{\alpha^\ast_s})^\top D^2_x v^n)(s,X_s)\\
	&  - (b^{a^n}D_xv^n)(s,X_s)-\frac12\text{tr}(\sigma^{a^n}(\sigma^{a^n})^\top D^2_x v^n)(s,X_s)
	- f^{a^n}(s, X_s) \big]\,ds\\
	& + (D_x v^n\sigma)(s,X_s) \, dW_s	\,.
	\end{split}
	\end{equation*}
	The resulting object can be seen as a second order BSDE (2BSDE). Analysis of 2BSDEs goes beyond the scope of this paper.
\end{remark}
\begin{remark}
Let us briefly consider the infinite-time-horizon control problem. 
In this case we consider a constant $\lambda > 0$ and the gain functional:
\begin{equation*}
	J(x,\alpha)=\E\left[\int_0^{\infty}e^{-\lambda s}f^{\alpha}(X^x_s)\,ds\right]\,.
\end{equation*}
It is known that the Bellman PDE for the value function is
\begin{equation*}
	\lambda v-\frac{1}{2}\text{tr}(\sigma\sigma^\top D_x^2 v)-\sup_{a\in A}\left(b^aD_xv+f^a\right)=0\,\,\,\text{on} \,\,\, \R^d\,.
\end{equation*}
The linear PDE from the iteration of the policy improvement algorithm then is
\begin{equation*}
	\lambda v^n-\frac{1}{2}\text{tr}(\sigma\sigma^\top D_x^2 v^n)-b^{a^n}D_x v^{n}-f^{a^n}=0\,\,\,\text{on} \,\,\, \R^d\,,
\end{equation*}
	where
\begin{equation*}
	a^n(x)=\arg\max_{a \in A}\left(b^a(x)D_xv^{n-1}(x)+f^a(x)\right)\,.
\end{equation*}
After applying It\^o's formula we get 
\begin{equation*}
\begin{split}
	dv^n(X_s) = & \left[ \frac12\text{tr}(\sigma\sigma^\top D^2_x v^n)(X_s) + (b^{\alpha^\ast_s} D_x v^n)(X_s) \right]\,ds + (D_x v^n\sigma) (X_s) \, dW_s\\
	= & \big[(b^{\alpha^\ast_s} D_x v^n)(X_s)  - (b^{a^n}D_xv^n)(X_s)- f^{a^n}(X_s) +\lambda v^n(X_s) \big]\,ds\\
	& + (D_x v^n\sigma)(X_s) \, dW_s	\,.
\end{split}
\end{equation*}
Let $Y^n_t:=v^n(X_t)$ and 	$Z^n_t:=\sigma(X_t)D_xv^n(X_t)$. Then after change of measure we may write
\begin{equation}
\label{eq: BSDE in infinite horizon}
	dY^n_s=-\big[(b^{a(X_s,Z^{n-1}_s)}\sigma^{-1})(X_s)Z^n_s+ f^{a(X_s,Z^{n-1}_s)}(X_s) -\lambda Y^n_s \big]\,ds + Z^n_s \, d\widehat{W}_s\,.
\end{equation} 
Let
\begin{equation*}
F_s(z,Z):=(b^{a(X_s,z)}\sigma^{-1})(X_s)Z+ f^{a(X_s,z)}(X_s)\,.
\end{equation*}
Hence~\eqref{eq: BSDE in infinite horizon} becomes
\begin{equation}
\label{eq: BSDE in infinite horizon 2}
dY^n_s=\left[-F_s(Z^{n-1}_s,Z^n_s) +\lambda Y^n_s \right]ds + Z^n_s \, d\widehat{W}_s\,,\,\,\,s\in [0,\infty)\,.
\end{equation}
To proceed, we need a suitable contraction-type inequality for this infinite time horizon BSDE.
Buckdahn and Peng~\cite{Buckdahn1999Ergodic} studied infinite time horizon BSDEs and have proved existence and uniqueness of their solutions for sufficiently large values of $\lambda$. 
To get the required contraction-type inequality we can use similar calculations as in Fuhrman and Tessitore~\cite[Theorems 3.2 and 3.7]{Fuhrman2004Infinite}, where they use Banach's fixed point theorem to show existence and uniqueness of solutions to infinite time horizon BSDEs.
Hence, for sufficiently large $\lambda > 0$ we would obtain results analogous to Theorem~\ref{main theorem alg 2} as well as the other theorems in the article.

\end{remark}

\section{Policy improvement}
\label{sec policy improvement}

We want to show that the policy obtained at each step of Algorithm~\ref{alg pia}
is an improvement on the one from the previous step. 
This is formulated as Theorem~\ref{thm pia improves} below. 
Note that we do not require Assumption~\ref{ass extra for pia} here.
\begin{theorem}
\label{thm pia improves}
Let Assumptions~\ref{assumption controlled SDE}, \ref{assumption rewards}, and \ref{ass convexity} hold. 
Assume that there exists $K\ge 0$ such that $\forall t\in[0,T]$, $\forall x\in\R^d$, and $\forall a\in A$
\[
|b^{a}\sigma^{-1}(t,x)|<K. 
\] 
Fix $n \in \mathbb N$. Let $v^n$ and $v^{n+1}$ be the solutions of~\eqref{eq policy iteration pde} at 
steps $n$ and $n+1$ of the algorithm.
Then for all $t\in[0,T]$, $x\in\R^d$ it holds that
\[
v^{n+1}(t,x)\ge v^{n}(t,x).
\]
\end{theorem}
\begin{proof}
Let $X = X^{t,x,\alpha^\ast}$ be the solution to the SDE~\eqref{sde} started from $(t,x)$ and controlled by the optimal control process $\alpha^\ast$; see Remark~\ref{remark alpha star}. 
Then, as in the proof of Theorem~\ref{main theorem alg 2}, we get that for $k=n,n+1$ with 
$Y^k=Y^{k,t,x}=v^k(\cdot, X^{t,x,\alpha^\ast})$ and with $Z^k=Z^{k,t,x}=(\sigma D_x v^k)(\cdot, X^{t,x,\alpha^\ast})$ 
we have the BSDE representation
\begin{equation*}
Y_t^k= \xi+\int_{t}^{T}F_s(Z^{k-1}_s,Z^{k}_s)ds-\int_{t}^{T}Z^k_s\,d\widehat{W}_s\,,\,\,k=n,n+1,	
\end{equation*}
where 
\[
F_s(z,Z):=b^{a(s,X_s,z)}(s,X_s)\sigma^{-1}(s,X_s)Z+f^{a(s,X_s,z)}(s,X_s)\,.
\]
Let us denote for $s\in[t,T]$ and $z\in\R^d$
\[
\phi^2_s(z):=F_s(Z^{n}_s,z)\,\,\,\text{and }\,\,\, \phi^1_s(z):=F_s(Z^{n-1}_s,z)\,.
\]
Hence, notice that by the definition of the $a^{n+1}$ (see~\eqref{eq def anplus1}), we have for all $s\in [t,T]$ that
\begin{equation*}
\begin{split}
\phi^2_s(Z^{n}_s)=&F_s(Z^n_s,Z^n_s)=b^{a^{n+1}}(s,X_s)\sigma^{-1}(s,X_s)Z^n_s+f^{a^{n+1}}(s,X_s)\\
=&\max_{a\in A}\left((b^aD_xv^{n})(s,X_s)+f^a(s,X_s)\right)\geq (b^{a^n}D_xv^{n})(s,X_s)+f^{a^n}(s,X_s)\\
=&F_s(Z^{n-1}_s,Z^{n}_s)=\phi^1_s(Z^n_s).
\end{split}
\end{equation*}
Therefore by the comparison principle for BSDEs (see Lemma~\ref{comparison principle}), we get
\[
Y_{t'}^{n+1}\ge Y^{n}_{t'}\,\,\, \forall t'\in [t,T].
\]
Hence, we have
\[
v^{n+1}(t,x) = Y^{n+1,t,x}_t \ge  Y^{n,t,x}_t = v^n(t,x)\,.
\]
\end{proof}

\begin{remark}
It is perhaps interesting to note that the comparison principle for BSDEs cannot be used to deduce that in the gradient iteration algorithm we have 
an ``improvement'' at each step.
Indeed, let us write the BSDE representation of the two steps of gradient iteration for $n,n+1\in\mathbb{N}$,
\begin{equation*}
Y_t^n= \xi+\int_{t}^{T} F_s(Z^{n-1}_s)ds-\int_{t}^{T}Z^n_s\,d\widehat{W}_s\,
\end{equation*}
and
\begin{equation*}
Y_t^{n+1}= \xi+\int_{t}^{T} F_s(Z^{n}_s)ds-\int_{t}^{T}Z^{n+1}_s\,d\widehat{W}_s\,,
\end{equation*}
where 
\[
F_s(z) := b^{a(s,X_s,z)}(s,X_s)\sigma^{-1}(s,X_s) z + f^{a(s,X_s,z)}(s,X_s).
\]
In order to apply a comparison principle for BSDEs (see Lemma~\ref{comparison principle}), we would need to have $F_s(Z^{n-1}_s)\le F_s(Z^{n}_s)$.
Nevertheless we observe that
\begin{equation*}
\begin{split}
F_s(Z^{n-1}_s)&=b^{a(s,X_s,Z^{n-1}_s)}\sigma^{-1}(s,X_s)Z^{n-1}_s+f^{a(s,X_s,Z^{n-1}_s)}\\
&=b^{a^n}(s,X_s)\sigma^{-1}(s,X_s)Z^{n-1}_s+f^{a^n}(s,X_s)
\\&=\max_{a\in A}(b^{a}(s,X_s)D_xv^{n-1}(s,X_s)+f^a(s,X_s)).
\end{split}
\end{equation*}
Similarly,
\begin{equation*}
\begin{split}
F_s(Z^{n}_s)&=b^{a(s,X_s,Z^{n}_s)}\sigma^{-1}(s,X_s)Z^{n}_s+f^{a(s,X_s,Z^{n}_s)}\\
&=b^{a^{n+1}}(s,X_s)\sigma^{-1}(s,X_s)Z^{n}_s+f^{a^{n+1}}(s,X_s)
\\&=\max_{a\in A}(b^{a}(s,X_s)D_xv^{n}(s,X_s)+f^a(s,X_s)).
\end{split}
\end{equation*}
From the above calculations we have no way to conclude that $F_s(Z^{n-1}_s)\le F_s(Z^{n}_s)$.
Thus the gradient iteration algorithm is not guaranteed to be improving the policy with each step.
\end{remark}

\section{Stability under Perturbations to Solution of the Linear PDE}
\label{sec stability lin PDE}
In this section we study a stability property of the policy improvement algorithm
under perturbations to solutions of the linear PDE~\eqref{eq policy iteration pde}
since in practical applications one will only solve this equation approximately.
Of course the maximization step~\eqref{eq def anplus1} of Algorithm~\ref{alg pia} can
now be performed only with this approximate solution, thus feeding the errors into further iterations.

Let $\varepsilon$ be a parameter (or a set of parameters), 
which determines the accuracy of our approximation to the solution of the linear PDE~\eqref{eq policy iteration pde}.
Let $\pi^n_\varepsilon$ be the policy at iteration $n$ obtained from an approximate
solution to the linear PDE. 
Let $v_\varepsilon^n$ denote the solution to
\begin{equation}
\label{eq pde for perturbations}
\begin{split}
\partial_t v_\varepsilon^n + \frac{1}{2}\text{tr}(\sigma \sigma^\top D^2_x v_\varepsilon^n)+b^{\pi^n_\varepsilon} D_x v^n_\varepsilon + f^{\pi^n_\varepsilon} & = 0\text{ on }[0,T)\times\R^d\,,\\
v_\varepsilon^n (T,\cdot) & = g\,\,\text{ on } x\in\R^d\,.
\end{split}
\end{equation}
At step $n$ of Algorithm~\ref{alg pia} we approximate the solution to the  
equation above (this is PDE~\eqref{eq policy iteration pde} but with $\pi^n_\varepsilon$ replacing $a^n$ everywhere).
We will denote such approximation by $\tilde v^n_\varepsilon$. 
The policy function for the next iteration step is then given by
\[
\pi^{n+1}_\varepsilon(t,x)=a(t,x,(\sigma D_x\tilde{v}^n_\varepsilon)(t,x))=
\text{arg}\max_{a \in A} \left[(b^a D_x\tilde{v}^n_\varepsilon)(t,x) + f^a(t,x) \right]\,,
\]
recalling that the function $a=a(t,x,z)$ was defined in~\eqref{control function}. 
We need to assume that $(t,x)\mapsto D_x \tilde v_\varepsilon^n$ is bounded so that $\pi_\varepsilon^{n+1}$ is Lipschitz in $x$
so that the solution to~\eqref{eq pde for perturbations} is $C^{1,2}([0,T]\times \mathbb R^d)$.
This assumption is not really a restriction as we know that the gradient of
the value function is bounded under our assumptions; see Krylov~\cite[Chapter 4, section 1, Theorem 1]{krylov controlled} and also Remark~\ref{remark alpha star}.
Any reasonable approximation should retain this property.

\begin{theorem}
Let Assumptions \ref{assumption controlled SDE}, \ref{assumption rewards}, \ref{ass convexity}, and \ref{ass extra for pia} hold. 
Let $(v^n)_{n\in \mathbb N}$ be the approximation sequence given by Algorithm~\ref{alg pia}. 
Let $(v^n_\varepsilon)_{n\in \mathbb N}$ be the approximation sequence given by~\eqref{eq pde for perturbations}.
Let $\alpha^*$ and $X^{t,x,\alpha^*}$ be the optimal control process for~\eqref{cp} 
and the associated diffusion started from $(t,x) \in  [0,T] \times \R^d$.
Assume that $D_x \tilde v_\varepsilon^n$ is uniformly bounded.
Define 
\[
\begin{split}
E^{k}_{t,x} := \bigg\|\Big[(\sigma (D_x v^k_\varepsilon - D_x \tilde v^k_\varepsilon))(\cdot, X^{t,x,\alpha^*})\Big]
\mathcal E^{-1/2}((b^{\alpha^*}\sigma^{-1})(\cdot,X^{t,x,\alpha^*}) \bullet W)_T
\bigg \|_{\mathbb H^2}\,.	
\end{split}
\]
Then there is $q\in(0,1)$ and $\gamma > 0$, depending only on $K, \theta, T$, such that for all $(t,x)\in[0,T]\times\R^d$  there exists $C=C(t,x)$ such that
\begin{equation*}
\begin{split}
& |v^n(t,x) -  v^n_\varepsilon(t,x)|^2 \le  C(t,x)q^{n}+2e^{\gamma (T-t)}\sum_{k=1}^nq^kE^{n-k}\,.	
\end{split}
\end{equation*}
\end{theorem}

\begin{proof} Let $X = X^{t,x,\alpha^\ast}$ be the solution to the SDE~\eqref{sde} started from $(t,x)$ and controlled by the optimal control process $\alpha^\ast$; see Remark~\ref{remark alpha star}.
By applying It\^o's formula  to $v^n_\varepsilon$ we get
\begin{equation*}
\begin{split}
dv^n_\varepsilon(s,X_s) = & \left[\partial_t v^n_\varepsilon(s,X_s) + \frac{1}{2}\text{tr}(\sigma\sigma^\top D^2_xv^n_\varepsilon)(s,X_s) + (b^{\alpha_s^*} D_x v^n_\varepsilon)(s,X_s) \right]\,ds \\ 
& + (D_x v^n_\varepsilon \sigma ) (s,X_s) \, dW_s\\
= & \left[(b^{\alpha_s^*} D_x v^n_\varepsilon)(s,X_s) - \left(b^{\pi^n_\varepsilon}D_x v^n_\varepsilon + f^{\pi^n_\varepsilon}\right)(s,X_s)
\right]\,ds\\
& + (D_x v^n_\varepsilon \sigma )(s,X_s) \, dW_s.	\,
\end{split}
\end{equation*}
Let us denote
\[
Y^n_{t,\varepsilon}:=v^n_\varepsilon(t,X_t)\,,\,\,\,\,Z^n_{t,\varepsilon}:= \sigma(t,X_t)D_xv^n_\varepsilon(t,X_t)\,,\,\,\,\xi:=g(X_T)\,,
\]
\[
F_s(z,Z):=b^{a(s,X_s,z)}(s,X_s)\sigma^{-1}(s,X_s)Z+f^{a(s,X_s,z)}(s,X_s)
\]
and
\[
\tilde{Z}^{n-1}_{t,\varepsilon}:=\sigma(t,X_t)D_x\tilde{v}^{n-1}_\varepsilon(t,X_t),
\]
where $\tilde{v}^{n-1}_\varepsilon$ is an approximate solution to corresponding PDE. 
Then using this notation, we may write
\begin{equation*}
\begin{split}
Y^n_{t,\varepsilon}=\xi-\int_{t}^{T}\left[(b^{\alpha_s^*}\sigma^{-1})(s,X_s)Z^n_{s,\varepsilon} -F_s(\tilde{Z}^{n-1}_{s,\varepsilon},Z^n_{s,\varepsilon}) \right]ds-\int_{t}^{T}Z^n_{s,\varepsilon}\,dW_s\,.	
\end{split}
\end{equation*}
Let $\hat{\mathbb P}$ and $\widehat{W}$ be given by Lemma~\ref{lemma P hat}.
Then the above equation becomes
\begin{equation*}
Y^n_{t,\varepsilon}= \xi+\int_{t}^{T}F_s(\tilde{Z}^{n-1}_{s,\varepsilon},Z^n_{s,\varepsilon}) ds-\int_{t}^{T}Z^n_{s,\varepsilon}\,d\widehat{W}_s\,.
\end{equation*}
We want to study the difference of $(Y^n_{\varepsilon},Z^n_{\varepsilon})$ with $(Y^n,Z^n)$, where $(Y^n,Z^n)$ solves the BSDE~\eqref{eq pia bsde iter in P hat2}.
\begin{equation}
\begin{split}
\label{eq stability estimate}
\hat{\mathbb E}[e^{\gamma t}|Y^n_t-Y^n_{t,\varepsilon}|^2]+&\|Z^n-Z^n_{\varepsilon}\|_{\hat{\mathbb H}^2_\gamma}^2\le 2\hat{\mathbb E}[e^{\gamma t}|Y_t-Y^n_{t}|^2]+2\|Z-Z^n\|_{\hat{\mathbb H}^2_\gamma}^2\\
&+2\hat{\mathbb E}[e^{\gamma t}|Y_t-Y^n_{t,\varepsilon}|^2]+2\|Z-Z^n_{\varepsilon}\|_{\hat{\mathbb H}^2_\gamma}^2\,,
\end{split}
\end{equation}
where $(Y,Z)$ solves the BSDE~\eqref{eq: BSDE for value function policy iteration}. Due to~\eqref{lipschitz F policy iteration} and
\begin{equation*}
|F_s(Z_s,Z_s)-F_s(\tilde{Z}^{n-1}_{s,\varepsilon},Z^n_{s,\varepsilon})|\le (C\theta+\theta)|Z_s-\tilde{Z}^{n-1}_{s,\varepsilon}|+K|Z_s-Z^n_{s,\varepsilon}|\,,
\end{equation*}
we can apply Lemma~\ref{policy convergence estimate}. Hence, there is $\tilde{q}\in (0,1/2)$ and $\gamma >0$ such that
\begin{equation*}
\hat{\mathbb E}[e^{\gamma t}|Y_t-Y^n_{t}|^2]+\|Z-Z^n\|_{\hat{\mathbb H}^2_\gamma}^2\le \tilde{q}\|Z-Z^{n-1}\|_{\hat{\mathbb H}^2_\gamma}^2
\end{equation*}
and
\begin{equation*}
\hat{\mathbb E}[e^{\gamma t}|Y_t-Y^n_{t,\varepsilon}|^2]+\|Z-Z^n_{\varepsilon}\|_{\hat{\mathbb H}^2_\gamma}^2\le \tilde q \|Z-\tilde{Z}^{n-1}_{\varepsilon}\|_{\hat{\mathbb H}^2_\gamma}^2\,.
\end{equation*}
Therefore we continue the estimate~\eqref{eq stability estimate}
\begin{equation*}
\begin{split}
\hat{\mathbb E}&[e^{\gamma t}|Y^n_t-Y^n_{t,\varepsilon}|^2]+\|Z^n-Z^n_{\varepsilon}\|_{\hat{\mathbb H}^2_\gamma}^2\\
&\le 2\hat{\mathbb E}[e^{\gamma t}|Y_t-Y^n_{t}|^2]+2\|Z-Z^n\|_{\hat{\mathbb H}^2_\gamma}^2+2\hat{\mathbb E}[e^{\gamma t}|Y_t-Y^n_{t,\varepsilon}|^2]+2\|Z-Z^n_{\varepsilon}\|_{\hat{\mathbb H}^2_\gamma}^2\,\\
&\le 2\tilde q\|Z-Z^{n-1}\|_{\hat{\mathbb H}^2_\gamma}^2+2\tilde q \|Z-\tilde{Z}^{n-1}_{\varepsilon}\|_{\hat{\mathbb H}^2_\gamma}^2\\
&\le 2\tilde q\|Z-Z^{n-1}\|_{\hat{\mathbb H}^2_\gamma}^2+4\tilde q\|Z-Z^{n-1}_{\varepsilon}\|_{\hat{\mathbb H}^2_\gamma}^2+4\tilde q\|Z^{n-1}_{\varepsilon}-\tilde{Z}^{n-1}_{\varepsilon}\|_{\hat{\mathbb H}^2_\gamma}^2\\
&\le 2\tilde q^2\|Z-Z^{n-2}\|_{\hat{\mathbb H}^2_\gamma}^2+4\tilde q^2\|Z-\tilde{Z}^{n-2}_{\varepsilon}\|_{\hat{\mathbb H}^2_\gamma}^2+4\tilde q\|Z^{n-1}_{\varepsilon}-\tilde{Z}^{n-1}_{\varepsilon}\|_{\hat{\mathbb H}^2_\gamma}^2\\
&\le 2\tilde q^2\|Z-Z^{n-2}\|_{\hat{\mathbb H}^2_\gamma}^2+8\tilde q^2\|Z-Z^{n-2}_{\varepsilon}\|_{\hat{\mathbb H}^2_\gamma}^2+8\tilde q^2\|Z^{n-2}_{\varepsilon}-\tilde{Z}^{n-2}_{\varepsilon}\|_{\hat{\mathbb H}^2_\gamma}^2\\
&\qquad+4\tilde q\|Z^{n-1}_{\varepsilon}-\tilde{Z}^{n-1}_{\varepsilon}\|_{\hat{\mathbb H}^2_\gamma}^2\le\cdots\\
&\le 2\tilde q^n\|Z-Z^{0}\|_{\hat{\mathbb H}^2_\gamma}^2+2^{n+1}\tilde q^n\|Z-Z^0_\varepsilon\|_{\hat{\mathbb H}^2_\gamma}^2+\sum_{k=1}^n2^{k+1}\tilde q^k\|Z^{n-k}_{\varepsilon}-\tilde{Z}^{n-k}_{\varepsilon}\|_{\hat{\mathbb H}^2_\gamma}^2\,.
\end{split}
\end{equation*}
This concludes the proof of the theorem.
\end{proof}

\section{Stability under Perturbation of the Maximization}
\label{sec stability max}
In this section we study a stability property of the gradient iteration algorithm under perturbations to maximization procedure~\eqref{eq gia update control}. 
Let $\bar{v}^{n}$ be the solution to  corresponding PDE at iteration $n$ of the gradient iteration algorithm, where instead of obtaining the control function corresponding to the exact maximum
\[
a^n(t,x) 
= a(t,x,(\sigma D_x\bar{v}^{n})(t,x))
=\text{arg}\max_{a \in A}((b^a D_x\bar{v}^{n}+f^a)(t,x))\,
\]
we only solve this maximization problem approximately and so we are dealing with 
a control function of the form
\[
\bar{a}(t,x,(\sigma D_x\bar{v}^{n})(t,x))
:=a(t,x,(\sigma D_x\bar{v}^{n})(t,x))+\varepsilon(t,x,(\sigma D_x\bar{v}^{n})(t,x))\,,
\]
where the function $\varepsilon=\varepsilon(t,x,z)$ determines the accuracy of our approximation. 

\begin{theorem}
Let Assumptions \ref{assumption controlled SDE}, \ref{assumption rewards}, \ref{ass convexity}, and \ref{ass extra for pia} hold. 
Let $(v^n)_{n\in\mathbb{N}}$ be the approximation sequence given by Algorithm \ref{alg via}. Let $(\bar{v}^n)_{n\in\mathbb{N}}$ be the approximation sequence given by the perturbations to the maximization procedure and assume that $v^0 = \bar v^0$. 
Let $\alpha^\ast$ and $X^{t,x,\alpha^\ast}$ be the optimal control process for~\eqref{cp} and the associated diffusion started from $(t,x)\in[0,T]\times\R^d$. Define
\[
E^{k+1}_{t,x}:=\left\|\left[1+|D_x\bar{v}^k(\cdot,X^{t,x,\alpha^*})|\right]\mathcal{E}^{-1/2}((b^{\alpha^*}\sigma^{-1})(\cdot,X^{t,x,\alpha^\ast}) \bullet W)_T\right\|^2_{\mathbb{H}^2}\,,
\]
\[
\varepsilon^{k+1}=\sup_{(s,y)\in[t,T]\times\R^d}|\varepsilon(s,y,(\sigma D_x\bar{v}^k)(s,y))|^2\,.
\]
Then there is $q\in(0,1)$ and $\gamma>0$, depending only on $K,\theta,T$, such that for all $(t,x)\in[0,T]\times\R^d$ there exists $C=C(t,x)$ such that
\begin{equation*}
|\bar{v}^n(t,x)-v^n(t,x)|^2\le  C(t,x)q^n+ e^{\gamma (T-t)}\frac{2\theta}{C\theta+K+\theta}\sum_{k=1}^{n}q^{n-k+1}\varepsilon^{k}E_{t,x}^{k}\,.
\end{equation*}
\end{theorem}
\begin{proof} Let $X = X^{t,x,\alpha^\ast}$ be the solution to the SDE~\eqref{sde} started from $(t,x)$ and controlled by the optimal control process $\alpha^\ast$; see Remark~\ref{remark alpha star}. 
As in the proof of Theorem~\ref{main theorem alg 1} we can  write two BSDEs we get after the change of measure given by Lemma~\ref{lemma P hat}. 
The first BSDE arises from the perturbations of the maximization: 
\[
\bar{Y}_t^n=\xi+\int_{t}^{T}\bar{F}_s(\bar{Z}^{n-1}_s)\,ds-\int_{t}^T\bar{Z}_s^n\,d\widehat{W}_s\,,
\]
where
\[
\bar{F}_s(z)=b^{\bar{a}(s,X_s,z)}\sigma^{-1}(s,X_s)z+f^{\bar{a}(s,X_s,z)}(s,X_s)\,.
\]
The second BSDE arises from the gradient iteration algorithm with the maximization performed exactly:
\[
Y^n_t=\xi+\int_{t}^TF_s(Z^{n-1}_s)\, ds-\int_{t}^TZ^n_s\,d\widehat{W}_s\,,
\]
where
\[
F_s(z)=b^{a(s,X_s,z)}\sigma^{-1}(s,X_s)z+f^{a(s,X_s,z)}(s,X_s)\,.
\]
We want to study the difference of $(\bar{Y}^n,\bar{Z}^n)$ with $(Y^n,Z^n)$. Hence, notice that
\begin{equation}\label{eq: diff of bar Y^n-Y^n}
\begin{split}
e^{\gamma t}\hat\E|\bar{Y}^n_{t}-Y^n_{t}|^2&+\|\bar{Z}^n-Z^n\|^2_{\hat{\mathbb H}^2_\gamma}
\le 2e^{\gamma t}\hat\E|Y_{t}-Y^n_{t}|^2+2\|Z-Z^n\|^2_{\hat{\mathbb H}^2_\gamma}\\
&+2e^{\gamma t}\hat\E|Y_t-\bar{Y}^n_{t}|^2+2\|Z-\bar{Z}^n\|^2_{\hat{\mathbb H}^2_\gamma}\,,
\end{split}
\end{equation}
where $(Y,Z)$ solves~\eqref{eq: BSDE for value function}.
Therefore, since
\begin{equation*}
\begin{split}
|F_s(Z_s)-F_s(\bar{Z}^{n-1}_s)|\le (C\theta+K+\theta)|Z_s-\bar{Z}^{n-1}_s|\,,
\end{split}
\end{equation*}
we can apply Lemma \ref{lem grad itr pert max} so that there is $q\in(0,1)$ and $\gamma>0$ such that
\begin{equation}
\label{eq outcome of lemma pert max}
\begin{split}
e^{\gamma t}\hat\E|Y_t-\bar{Y}^n_{t}|^2+\|Z-\bar{Z}^n\|^2_{\hat{\mathbb H}^2_\gamma}
\le& q \|Z-\bar{Z}^{n-1}\|^2_{\hat{\mathbb H}^2_\gamma}\\
&+\frac{q}{C\theta+K+\theta}\|\bar{F}(\bar{Z}^{n-1})-F(\bar{Z}^{n-1})\|^2_{\hat{\mathbb H}^2_\gamma}.
\end{split}
\end{equation}
Now we need to estimate the second term of the right-hand side (RHS). Notice that by Assumption \ref{ass extra for pia} the following holds:
\begin{equation}
\label{eq diff of generators pert max} 
\begin{split}
&|\bar{F}_s(\bar{Z}^{n-1}_s)-F_s(\bar{Z}^{n-1}_s)|\\
&\le |\sigma^{-1}(s,X_s)\bar{Z}^{n-1}_s||b^{\bar{a}(s,X_s,\bar{Z}^{n-1}_s)}(s,X_s)-b^{a(s,X_s,\bar{Z}^{n-1}_s)}(s,X_s)|\\
&\qquad+|f^{\bar{a}(s,X_s,\bar{Z}^{n-1}_s)}(s,X_s)-f^{a(s,X_s,\bar{Z}^{n-1}_s)}(s,X_s)|\\
&\le \sqrt{\theta}|\sigma^{-1}(s,X_s)\bar{Z}^{n-1}_s||\varepsilon(s,X_s,\bar{Z}^{n-1}_s)|+\sqrt{\theta}|\varepsilon(s,X_s,\bar{Z}^{n-1}_s)|\,.
\end{split}
\end{equation}
Hence by~\eqref{eq diff of generators pert max} we have
\begin{equation}
\label{eq estimate for pert max}
\|\bar{F}(\bar{Z}^{n-1})-F(\bar{Z}^{n-1})\|^2_{\hat{\mathbb H}^2_\gamma}\le \theta\|(1+|\sigma^{-1}(\cdot,X)\bar{Z}^{n-1}|)\varepsilon(\cdot,X,\bar{Z}^{n-1})\|^2_{\hat{\mathbb H}^2_\gamma}\,.
\end{equation}
By inequalities~\eqref{eq: diff of bar Y^n-Y^n},~\eqref{eq outcome of lemma pert max},~\eqref{eq estimate for pert max}, and the result of Theorem~\ref{main theorem alg 1} and since $\bar{Y}^{t,x,n}_t=\bar{v}^n(t,x),\,Y^{t,x,n}_t=v^n(t,x)$ 
as well as $Z^0 = \bar Z^0$, we conclude that
\begin{equation*}
\begin{split}
e^{\gamma t}&|\bar{v}^n(t,x)-v^n(t,x)|^2
\\
&\le 2e^{\gamma t}\hat{\E}|Y_t-Y^n_t|^2+\|Z-Z^n\|^2_{\hat{\mathbb H}^2_\gamma}+2e^{\gamma t}\hat{\E}|Y_t-\bar{Y}^n_t|^2+\|Z-\bar{Z}^n\|^2_{\hat{\mathbb H}^2_\gamma}\\
& \le 2q\|Z-Z^{n-1}\|^2_{\hat{\mathbb H}^2_\gamma}+2q\|Z-\bar{Z}^{n-1}\|^2_{\hat{\mathbb H}^2_\gamma}+\frac{2q}{C\theta+K+\theta}\|\bar{F}(\bar{Z}^{n-1})-F(\bar{Z}^{n-1})\|^2_{\hat{\mathbb H}^2_\gamma}\\
&\le 4q^n\|Z-Z^0\|^2_{\hat{\mathbb H}^2_\gamma}+ \sum_{k=1}^{n}q^{k}\frac{2\theta}{C\theta+K+\theta}\|(1+|\sigma^{-1}(\cdot,X)\bar{Z}^{n-k}|)\varepsilon(\cdot,X,\bar{Z}^{n-k})\|^2_{\hat{\mathbb H}^2_\gamma}\,.
\end{split}
\end{equation*}
\end{proof}

We obtain the same result for the policy improvement algorithm.

\begin{theorem}
Let Assumptions \ref{assumption controlled SDE}, \ref{assumption rewards}, \ref{ass convexity}, and \ref{ass extra for pia} hold. 
Let $(v^n)_{n\in\mathbb{N}}$ be the approximation sequence given by Algorithm \ref{alg pia}. Let $(\bar{v}^n)_{n\in\mathbb{N}}$ be the approximation sequence given by the perturbations to the maximization procedure. 
Let $\alpha^\ast$ and $X^{t,x,\alpha^\ast}$ be the optimal control process for~\eqref{cp} and the associated diffusion started from $(t,x)\in[0,T]\times\R^d$. Define
\[
E^{k}_{t,x}:=\left\|\left[1+D_x\bar{v}^k(\cdot,X^{t,x,\alpha^\ast})\right]\mathcal{E}^{-1/2}((b^{\alpha^*}\sigma^{-1})(\cdot,X^{t,x,\alpha^\ast}) \bullet W)_T\right\|^2_{\mathbb{H}^2}\,,
\]
\[
\varepsilon^{k+1}=\sup_{(s,y)\in[t,T]\times\R^d}|\varepsilon(s,y,(\sigma D_x\bar{v}^k)(s,y))|^2\,.
\]
Then there is $q\in(0,1)$ and $\gamma>0$, depending only on $K,\theta,T$, such that for all $(t,x)\in[0,T]\times\R^d$ there is $C=C(t,x)$ such that
\begin{equation*}
|\bar{v}^n(t,x)-v^n(t,x)|^2\le C(t,x)q^n+ \frac{2\theta}{\max(C\theta+\theta,K)}e^{\gamma (T-t)} \sum_{k=1}^{n}q^{k}\varepsilon^{n-k+1}E_{t,x}^{n-k+1}\,.
\end{equation*}
\end{theorem}
\begin{proof} Let $X = X^{t,x,\alpha^\ast}$ be the solution to the SDE~\eqref{sde} started from $(t,x)$ and controlled by the optimal control process $\alpha^\ast$; see Remark~\ref{remark alpha star}. 
Due to Theorem \ref{main theorem alg 2} we can  write two BSDEs we get after the change of measure: first from the perturbation and second from the gradient iteration
\[
\bar{Y}_t^n=\xi+\int_{t}^{T}\bar{F}_s(\bar{Z}^{n-1}_s,\bar{Z}_s^n)\,ds-\int_{t}^T\bar{Z}_s^n\,d\widehat{W}_s\,,
\]
\[
Y^n_t=\xi+\int_{t}^TF_s(Z^{n-1}_s,Z^n_s)\, ds-\int_{t}^TZ^n_s\,d\widehat{W}_s\,,
\]
where
\[
\bar{F}_s(z,Z):=b^{\bar{a}(s,X_s,z)}(s,X_s)\sigma^{-1}(s,X_s)Z+f^{\bar{a}(s,X_s,z)}(s,X_s)\,,
\]
\[
F_s(z,Z):=b^{a(s,X_s,z)}(s,X_s)\sigma^{-1}(s,X_s)Z+f^{a(s,X_s,z)}(s,X_s)\,.
\]
Similarly, we want to study the difference of $(\bar{Y}^n,\bar{Z}^n)$ with $(Y^n,Z^n)$. Hence, notice that
\begin{equation}\label{eq: difference of Y bar and Y^n policy iteration}
\begin{split}
e^{\gamma t}\E|\bar{Y}^n_{t}-Y^n_{t}|^2&+\|\bar{Z}^n-Z^n\|^2_{\hat{\mathbb H}^2_\gamma}
\le 2e^{\gamma t}\E|Y_{t}-Y^n_{t}|^2+\|Z-Z^n\|^2_{\hat{\mathbb H}^2_\gamma}\\
&+2e^{\gamma t}\E|Y_t-\bar{Y}^n_{t}|^2+\|Z-\bar{Z}^n\|^2_{\hat{\mathbb H}^2_\gamma}\,,
\end{split}
\end{equation}
where $(Y,Z)$ solves~\eqref{main theorem alg 2}. Therefore, since
\begin{equation*}
\begin{split}
|F_s(Z_s,Z_s)-F_s(\bar{Z}^{n-1}_s,\bar{Z}^{n}_s)|\le \theta C|Z_s-\bar{Z}^{n-1}_s|+K|Z_s-\bar{Z}^{n}_s|+\theta|Z_s-\bar{Z}^{n-1}_s|\,,
\end{split}
\end{equation*}
we can apply Lemma \ref{lem policy impr pert max} so that there is $q\in(0,1)$ and $\gamma>0$ such that
\begin{equation}
\label{eq outcome of lemma pert max policy iteration}
\begin{split}
e^{\gamma t}\E|Y_t-\bar{Y}^n_{t}|^2&+\|Z-\bar{Z}^n\|^2_{\hat{\mathbb H}^2_\gamma}
\le q \|Z-\bar{Z}^{n-1}\|^2_{\hat{\mathbb H}^2_\gamma}\\
&+\frac{q}{\max(C\theta+\theta,K)}\|\bar{F}(\bar{Z}^{n-1},\bar{Z}^n)-F(\bar{Z}^{n-1},\bar{Z}^n)\|^2_{\hat{\mathbb H}^2_\gamma}.
\end{split}
\end{equation}
Now we need to estimate the second term of the RHS. Notice that by Assumption \ref{ass extra for pia} we have that
\begin{equation}
\label{eq diff of generators pert max policy iteration} 
\begin{split}
&|\bar{F}_s(\bar{Z}^{n-1}_s,\bar{Z}^{n}_s)-F_s(\bar{Z}^{n-1}_s,\bar{Z}^{n}_s)|\\
&\le\left|b^{\bar{a}(s,X_s,\bar{Z}^{n-1}_s)}(s,X_s)\sigma^{-1}(s,X_s)\bar{Z}^{n}_s-b^{a(s,X_s,\bar{Z}^{n-1}_s)}(s,X_s)\sigma^{-1}(s,X_s)\bar{Z}^{n}_s\right|\\
 &\qquad+\left|f^{\bar{a}(s,X_s,\bar{Z}^{n-1}_s)}(s,X_s)-f^{a(s,X_s,\bar{Z}^{n-1}_s)}(s,X_s)\right|\\
 &\le \sqrt{\theta}|\sigma^{-1}(s,X_s)\bar{Z}^{n}_s||\varepsilon(s,X_s,\bar{Z}^{n-1}_s)|+\sqrt{\theta}|\varepsilon(s,X_s,\bar{Z}^{n-1}_s)|\,.
\end{split}
\end{equation}
Hence by~\eqref{eq diff of generators pert max policy iteration} we have
\begin{equation}
\label{eq estimate for pert max policy iteration}
\|\bar{F}(\bar{Z}^{n-1},\bar{Z}^n)-F(\bar{Z}^{n-1},\bar{Z}^n)\|^2_{\hat{\mathbb H}^2_\gamma}\le \theta\|(1+|\sigma^{-1}(s,X)\bar{Z}^{n}|)\varepsilon(\cdot,X,\bar{Z}^{n-1})\|^2_{\hat{\mathbb H}^2_\gamma}\,.
\end{equation}
By inequalities~\eqref{eq: difference of Y bar and Y^n policy iteration},~\eqref{eq outcome of lemma pert max policy iteration},~\eqref{eq estimate for pert max policy iteration}, by the result of Theorem \ref{main theorem alg 2}, and by $\bar{Y}^{t,x,n}_t=\bar{v}^n(t,x)$, $Y^{t,x,n}_t=v^n(t,x)$ we conclude that
\begin{equation*}
\begin{split}
e^{\gamma t}&|\bar{v}^n(t,x)-v^n(t,x)|^2
\\
&\le 2e^{\gamma t}\hat{\E}|Y_t-Y^n_t|^2+\|Z-Z^n\|^2_{\hat{\mathbb H}^2_\gamma}+2e^{\gamma t}\hat{\E}|Y_t-\bar{Y}^n_t|^2+\|Z-\bar{Z}^n\|^2_{\hat{\mathbb H}^2_\gamma}\\
& \le 2q\|Z-Z^{n-1}\|^2_{\hat{\mathbb H}^2_\gamma}+2q\|Z-\bar{Z}^{n-1}\|^2_{\hat{\mathbb H}^2_\gamma}\\
&\qquad+\frac{2q}{\max(C\theta+\theta,K)}\|\bar{F}(\bar{Z}^{n-1})-F(\bar{Z}^{n-1})\|^2_{\hat{\mathbb H}^2_\gamma}\\
&\le 2q^n\|Z-Z^0\|^2_{\hat{\mathbb H}^2_\gamma}+2q^n\|Z-\bar Z^0\|^2_{\hat{\mathbb H}^2_\gamma}\\
&\qquad+ \sum_{k=1}^{n}q^{k}\frac{2\theta}{\max(C\theta+\theta,K)}\|(1+|\sigma^{-1}(\cdot,X)\bar{Z}^{n-k+1}|)\varepsilon(\cdot,X,\bar{Z}^{n-k})\|^2_{\hat{\mathbb H}^2_\gamma}\,.
\end{split}
\end{equation*}
\end{proof}

\section{Example}\label{sec example}
In this section we would like to consider an example when Assumptions~\ref{assumption controlled SDE}, \ref{assumption rewards},
\ref{ass convexity}, and~\ref{ass extra for pia} hold. Let $t\mapsto s(t)$ and $t\mapsto k(t)$ be continuous functions for $t\in[0,T]$. Consider the state which is governed by the controlled SDE
\begin{equation*}
dX_t=s(t)\sin \alpha_t\,dt+\sqrt{2}\,dW_t\,,\,\,\,t\in[0,T]\,,
\end{equation*}
and consider the cost functional
\begin{equation*}
J(t,x,\alpha)=\E\left[\int_t^Tk(s)\cos \alpha_s\,ds+g(X_T)\right]\,.
\end{equation*}
The aim is to maximize $J$ over admissible controls $\alpha\in\mathcal{A}$.
The value function $v=\sup_{\alpha\in\mathcal{A}}J(t,x,\alpha)$ satisfies the Bellman PDE
\begin{equation*}
\partial_t v+D_x^2v+\sup_{a\in A}\left[s(t)\sin a\, D_xv+k(t)\cos a\right]=0\,,\,\,\,\text{on}\,\,[0,T)\times\R\,,
\end{equation*}
with the terminal condition $v(T,x)=g(x):=\arctan(x)$. Hence, the optimal control is
\begin{equation*}
a(t,x)=\arctan\left(\frac{s(t)D_xv}{k(t)}\right)\,.
\end{equation*}
It is easy to check that Assumptions~\ref{assumption controlled SDE}, \ref{assumption rewards},
\ref{ass convexity}, and~\ref{ass extra for pia} hold for this problem.
Therefore, the Bellman PDE becomes
\begin{equation}\label{eq: pde for value function example}
\partial_t v+D_x^2v+\frac{\frac{(s(t)D_xv)^2}{k(t)}}{\sqrt{1+\left(\frac{s(t)D_xv}{k(t)}\right)^2}} +\frac{k(t)}{\sqrt{1+\left(\frac{s(t)D_xv}{k(t)}\right)^2}}=0\,.
\end{equation}

We can solve this problem using the policy improvement algorithm by approximating the Bellman PDE with a sequence of linear PDEs:

\textit{Step} 1. Make an initial choice of control $a^0(t,x)$.

\textit{Step} 2. For $n=0,1,\dots$:
\begin{itemize}
	\item Evaluation step: Find a solution $v^n=v^n(t,x)$ to the linear PDE
	\begin{equation}\label{eq: pde for v^n example}
	\partial_t v^n+D_x^2 v^n+s(t)\sin a^nD_xv^n+\cos a^n=0\,.
	\end{equation}
	\item Improvement step: Find a new policy $a^{n+1}=a^{n+1}(t,x)$ such that
	\begin{equation*}
	a^{n+1}(t,x)=\arctan\left(\frac{s(t)D_xv^n}{k(t)}\right)\,.
	\end{equation*}
\end{itemize} 

\textit{Step} 3. Iterate the process until no changes occur in the controls updates.

One can do similar calculations in the case of the gradient iteration algorithm.

We will solve~\eqref{eq: pde for value function example} and~\eqref{eq: pde for v^n example} by the finite difference method. For simplicity, let us choose $s(t)=1$ and $k(t)=1$ for all $t\in[0,T]$. In Figure~\ref{plot for errors example}, one can see the logarithm of the error between the value function obtained by the iterative methods, by the policy improvement algorithm, and by the gradient iteration algorithm at every step and the value function obtained by the solution of the Bellman PDE. This shows the fast convergence of the policy improvement method for our example in one dimension. In Figure~\ref{figure from policy iteration}, we can see that after only a few steps the policies obtained from the policy improvement algorithm are close to the exact one. Finally, in Figure~\ref{figure from the true solution of the example}, we plot the value function and the policy from the solution of the Bellman PDE.

\begin{figure}[h]
\centering
\minipage{0.49\textwidth}
\includegraphics[width=\linewidth]{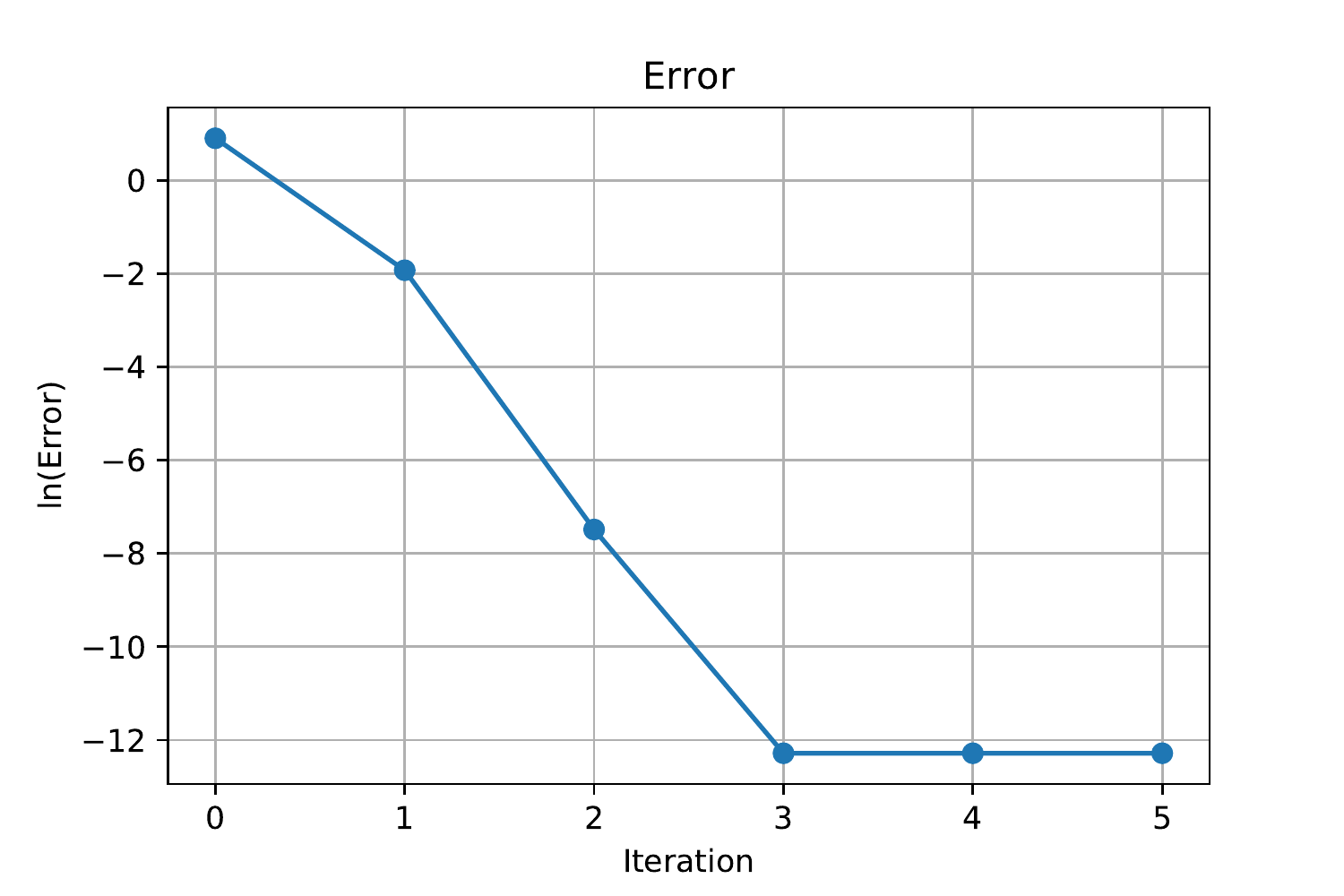}
\endminipage\hfill
\minipage{0.49\textwidth}
\includegraphics[width=\linewidth]{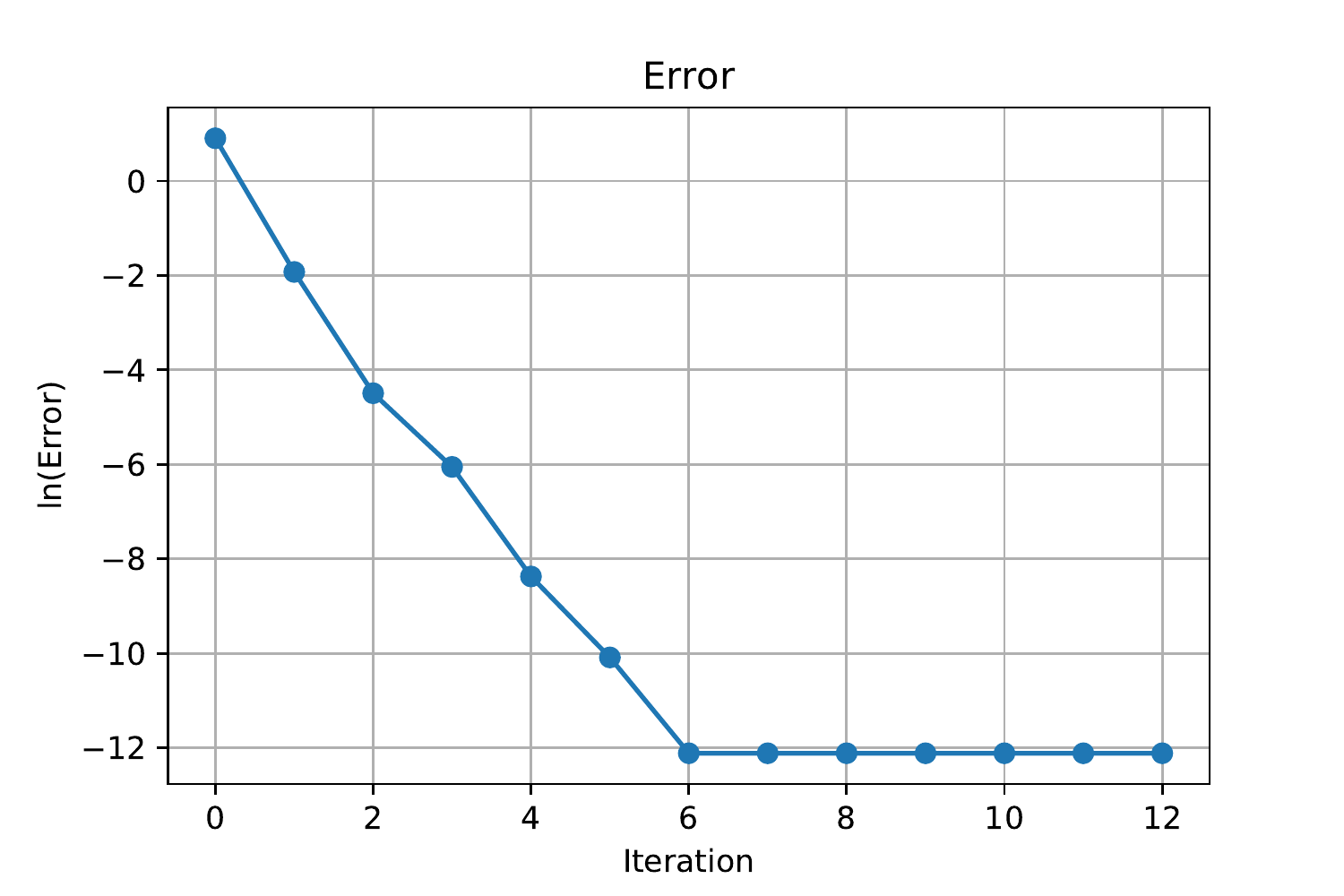}
\endminipage
\caption{\label{plot for errors example} Plot for the logarithm of the error between the value function obtained by the iterative methods, by the policy improvement, and by the gradient iteration algorithms (from left to right) at every step and the value function obtained by the solution of the Bellman PDE.
Note that the convergence stops once we have reached the accuracy of the finite-difference solver.
}	
\end{figure}


\begin{figure}[!htb]
	\minipage{0.32\textwidth}
	\includegraphics[width=\linewidth]{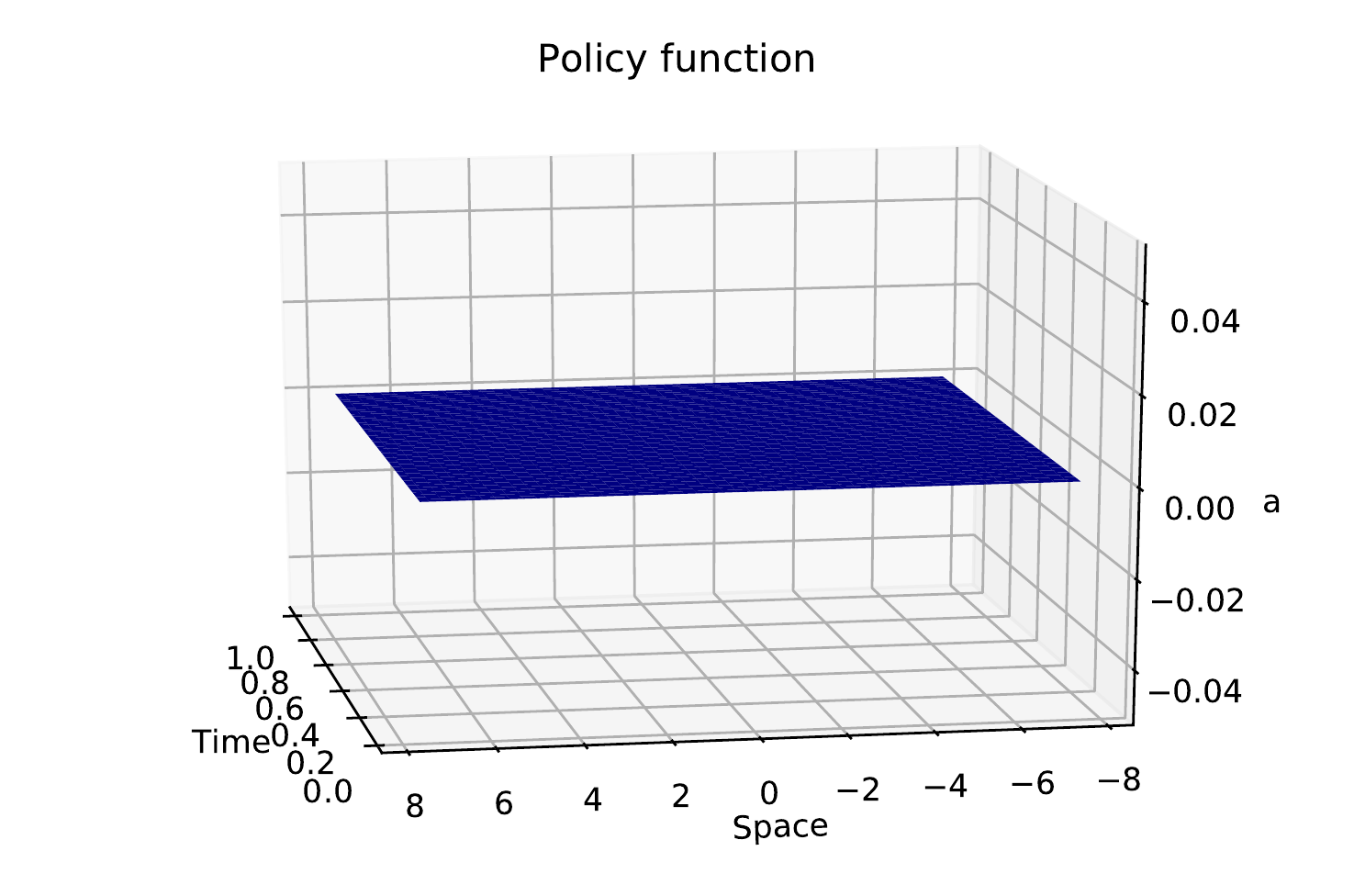}
	\endminipage\hfill
	\minipage{0.32\textwidth}
	\includegraphics[width=\linewidth]{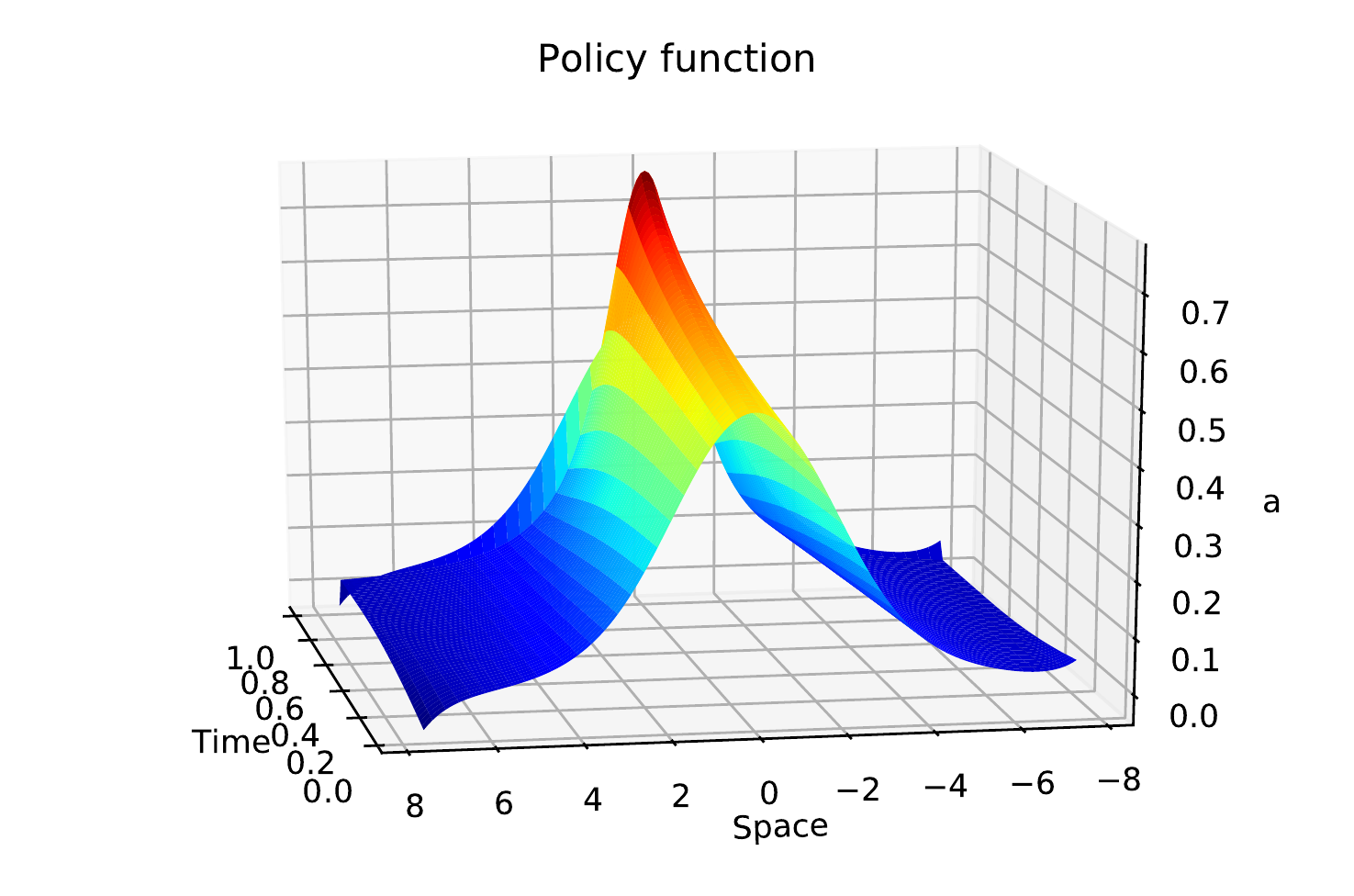}
	\endminipage\hfill
	\minipage{0.32\textwidth}%
	\includegraphics[width=\linewidth]{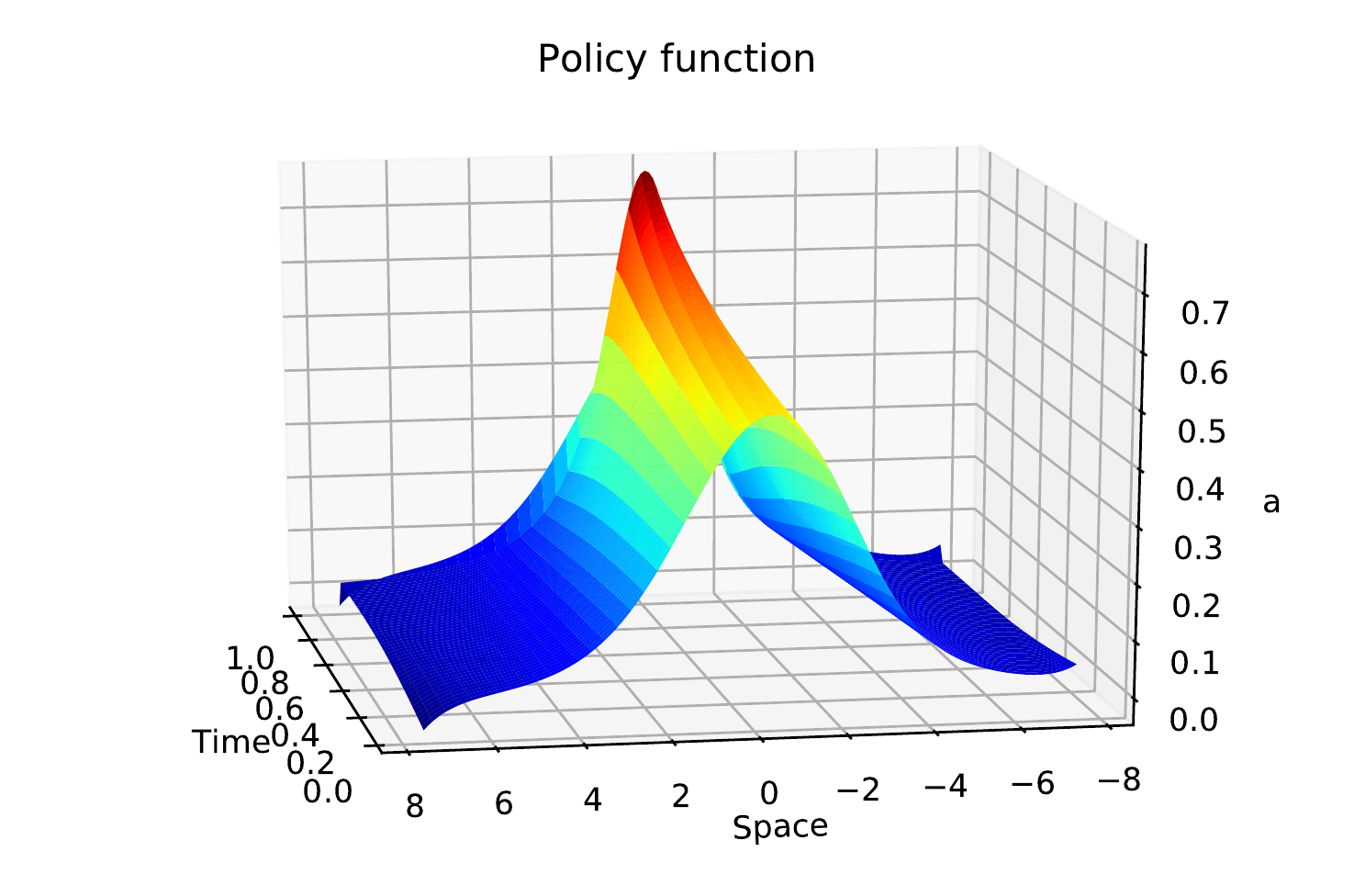}
	\endminipage
	\caption{\label{figure from policy iteration}Plot of the initial policy and policies obtained by the policy improvement algorithm at Steps 1 and 5}
\end{figure}

\begin{figure}[!htb]
	\minipage{0.49\textwidth}
	\includegraphics[width=\linewidth]{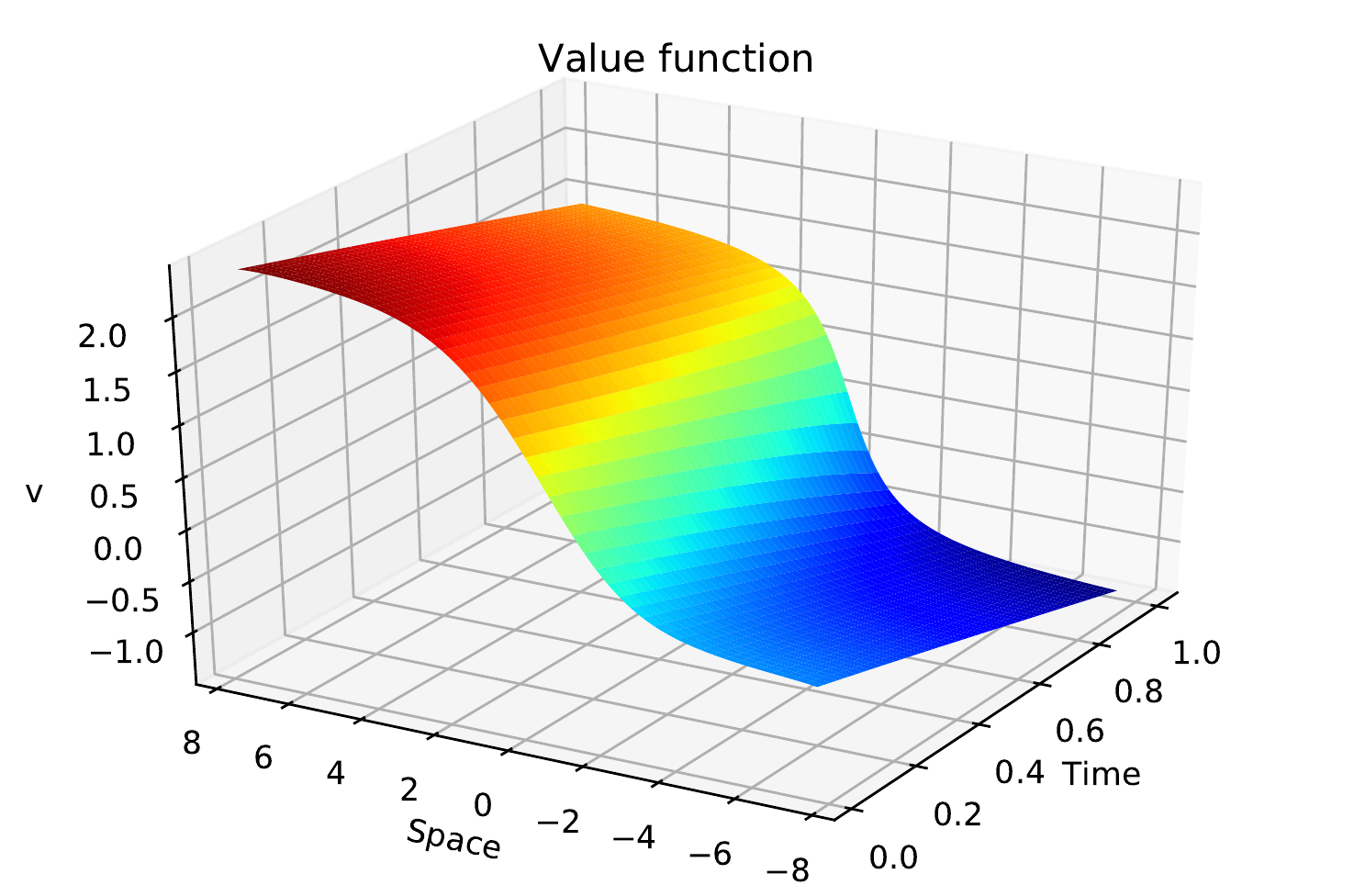}
	\endminipage\hfill
	\minipage{0.49\textwidth}
	\includegraphics[width=\linewidth]{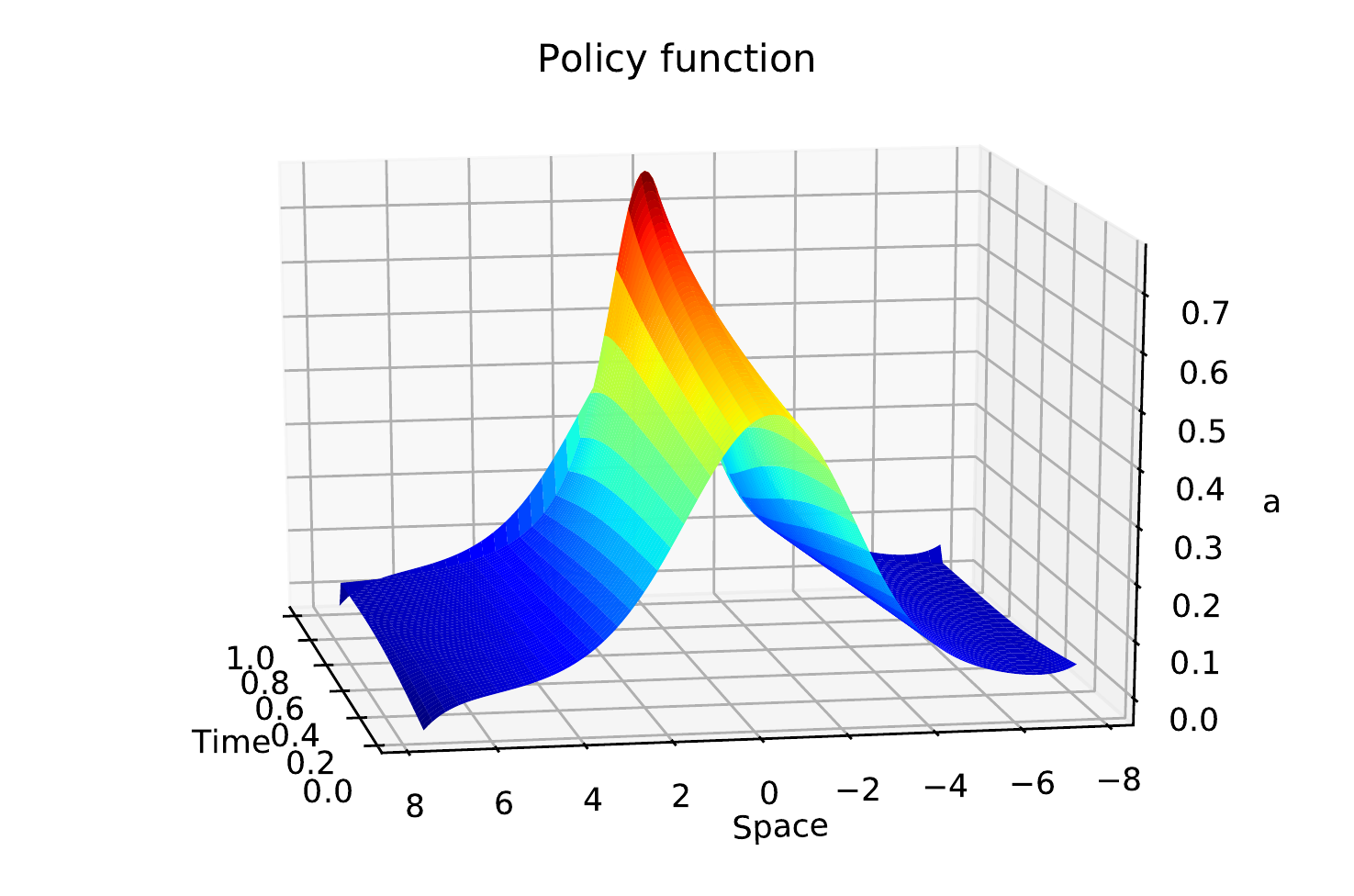}
	\endminipage\hfill
	\caption{\label{figure from the true solution of the example}Plot of the true value function and the policy}
\end{figure}

\appendix

\section{Some results from theory of BSDEs}
\label{sec BSDEs}

We fix a finite horizon $T\in (0,\infty)$.
We fix a filtered probability space $(\Omega, \mathcal{F}, \F=(\mathcal{F}_t)_{0\le t\le T}, \Prob)$.
Let there be a $d'$-dimensional Wiener martingale on this space. 

\begin{lemma}
\label{bsde well defined}
Let $F:\Omega \times [0,T] \times \mathbb R^d \to \mathbb R$ be a measurable function that
satisfies the following conditions: The process $(F_t(0))_{t\in [0,T]}$ is in $\mathbb H^2$.  
Moreover there is a constant $\theta>0$ such that   
\begin{equation*}
|F_t(z)| 
\leq |F_t(0)|+\theta|z|\,\,\,\forall z\in\R^d, t\in[0,T],\,\,a.s.
\end{equation*}
Then, for every $\xi\in L^2(\Omega, \mathcal F_T)$ and $z\in \mathbb{H}^2$,  there is a unique solution $(Y,Z) \in \mathcal S^2 \times \mathbb H^2$ to
\begin{equation}
\label{bsde in lemma}
Y_t=\xi+\int_{t}^{T}F_s(z_s)\,ds-\int_{t}^{T}Z_s\, d W_s,\,\,t\in[0,T],\,{\mathbb{P}}\text{-}\text{a.s}.	
\end{equation}
\end{lemma}
\begin{proof}
This follows immediately from, e.g., Pham~\cite[Theorem 6.2.1]{pham book}. 	
\end{proof}

\begin{lemma}
\label{lem classic bsde iteration}
Let $F:\Omega \times [0,T] \times \mathbb R^d \to \mathbb R$ satisfy the hypothesis of Lemma~\ref{bsde well defined}.
Fix $\xi \in L^2(\Omega, \mathcal F_T)$.
Let $\Phi : \mathbb{H}^2 \ni z\mapsto (Y,Z)\in  \mathcal S^2\times {\mathbb{H}}^2$, where $(Y,Z)$ is the unique solution to~\eqref{bsde in lemma}. Moreover assume that for $z^1, z^2\in \mathbb H^2$ the following condition satisfies that there is a constant $\theta>0$ such that   
\begin{equation}
\label{eq property of F in the lemma bsde well defined}
|F_t(z^1_t) - F_t(z^2_t)| 
\leq \theta|z^1_t-z^2_t|\,,\,\,\,t\in[0,T],\,\,a.s.
\end{equation}
Then there is $\gamma > 0$ and $q\in (0,1)$ such that for $(Y^i,Z^i) := \Phi(z^i)$, $i=1,2$, 
and any $t\in [0,T]$ we have
\begin{equation*}
\E\left[e^{\gamma t}|Y^1_t-Y^2_t|^2\right]+\|Z^1-Z^2\|^2_{\mathbb H^2_\gamma}\le q\|z^1-z^2\|^2_{\mathbb H^2_\gamma}\,.
\end{equation*}
\end{lemma}

The proof is well known and is included, e.g., as part of Pham~\cite[Proof of Theorem 6.2.1]{pham book}. 
We provide it here for the convenience of the reader and before we proceed
we need to make the following observation.
\begin{remark}
\label{rem stoch int unif int}
Assume that $Y\in \mathcal S^2$ and $Z\in \mathbb H^2$ and let 
\[
M_t:=\int_{0}^{t}e^{\gamma s} Z_s \, Y_s\, dW_s\,.
\]
Then $\sup_{t\le T}|M_t|\in L^1(\Omega, \mathcal F_T)$ and hence $M_t$ is a uniformly integrable martingale. 
Indeed, from the Burkholder--Davis--Gundy inequality and the Young inequality we get
\begin{equation*}
\begin{split}
\E\left[\sup_{t\le T}|M_t|\right]&\le C_1\E\left[\left(\int_{0}^{T}e^{2\gamma s}|Y_s|^2|Z_s|^2ds\right)^{1/2}\right]\\
&\le e^{\gamma T}C_1\E\left[\left(\sup_{t\le T}|Y_t|^2\int_{0}^{T}|Z_s|^2ds\right)^{1/2}\right]\\
&\le \frac{e^{\gamma T}}{2}C_1\E\left[\sup_{t\le T}|Y_t|^2+\int_{0}^{T}|Z_s|^2ds\right]<\infty.
\end{split}
\end{equation*}
\end{remark}

\begin{proof}[Proof of Lemma~\ref{lem classic bsde iteration}]
Consider $\gamma > 0$ which we will fix later. We denote  $\delta z:=z^1-z^2$, $\delta Z:=Z^1-Z^2$, $\delta Y:=Y^1-Y^2$, and $\delta F:=F(z^1)-F(z^2)$. We then apply It\^o's formula to $e^{\gamma t}|\delta Y_t|^2$:
\begin{equation*}
\begin{split}
e^{\gamma t}|\delta Y_t|^2+\int_t^T e^{\gamma s}|\delta Z_s|^2ds 
& =\int_{t}^{T}e^{\gamma s}(2\delta Y_s\, \delta F_s-\gamma |\delta Y_s|^2)ds\\&-2\int_{t}^{T}e^{\gamma s}\delta Z_s\,\delta Y_s \, dW_s\,.
\end{split}
\end{equation*}
Due to Remark~\ref{rem stoch int unif int}, the stochastic integral vanishes by taking expectation. 
Hence
\begin{equation*}
\begin{split}
\E\left[e^{\gamma t}|\delta Y_t|^2+\int_t^T e^{\gamma s}|\delta Z_s|^2ds\right] =\E\left[\int_{t}^{T}e^{\gamma s}(2\delta Y_s \, \delta F_s-\gamma |\delta Y_s|^2)\,ds\right]\,.
\end{split}
\end{equation*}	
By the Lipschitz property of the generator and by the Young inequality we continue our estimate, noting that for any $\varepsilon>0$, we have
\begin{equation*}
\begin{split}
\E\Bigg[e^{\gamma t}|\delta Y_t|^2+ & \int_t^T e^{\gamma s}|\delta Z_s|^2\,ds\Bigg]\le  \E\left[\int_{t}^{T}e^{\gamma s}(2\theta|\delta Y_s||\delta z_s|-\gamma |\delta Y_s|^2)ds\right]\\
&\le \E\left[\int_{t}^{T}e^{\gamma s}\left(\theta(\varepsilon|\delta Y_s|^2+\varepsilon^{-1}|\delta z_s|^2)-\gamma |\delta Y_s|^2\right)ds\right]\,.
\end{split}
\end{equation*}
Choose $\varepsilon$ such that $\gamma=\varepsilon\theta$. Thus
\begin{equation}
\begin{split}
\label{estimate in contraction value iteration}
\E\left[e^{\gamma t}|\delta Y_t|^2+\int_t^T e^{\gamma s}|\delta Z_s|^2ds\right]&\le \E\left[\int_{t}^{T}e^{\gamma s}\left(\theta\varepsilon^{-1}|\delta z_s|^2\right)ds\right]\le \frac{\theta^2}{\gamma} \|\delta z\|^2_{\mathbb H^2_\gamma}.
\end{split}
\end{equation}
Hence, from~\eqref{estimate in contraction value iteration} we have that for $\gamma>  \theta^2$ and any $t\in [0,T]$ 
\begin{equation*}
\E\left[e^{\gamma t}|Y^1_t-Y^2_t|^2\right]+\|Z^1-Z^2\|^2_{\mathbb H^2_\gamma}\le q\|z^1-z^2\|^2_{\mathbb H^2_\gamma}\,,
\end{equation*}
where $q=\frac{\theta^2}{\gamma}\in(0,1)$. This concludes the proof of the lemma.
\end{proof}

\begin{lemma}
\label{policy well defined}
Let $F: \Omega\times[0,T]\times \mathbb R^d\times \mathbb R^d \to \mathbb R$ be a measurable function such that the process $(F_t(0,0))_{t\in[0,T]}$ is in $\mathbb H^2$ and such that there are $\theta, K > 0$ so that for all $t\in [0,T]$, $z,Z\in \mathbb R^d$ we have
\[
|F_t(z,Z)| \leq|F_t(0,0)|+ \theta|z|+K|Z|\,\,\, {\text a.s.}
\]

If $\xi \in L^2(\Omega, \mathcal F_T)$ and $z\in \mathbb H^2$, then there is a unique solution $(Y,Z) \in \mathcal S^2 \times \mathbb H^2$ to
\begin{equation}
\label{eq fixed point pia}
Y_t=\xi + \int_{t}^{T}F_s(z_s,Z_s)\, ds - \int_{t}^{T}Z_s\,dW_s\,,\,\,\,t\in[0,T]\,,\,\,\mathbb{P}\text{-}a.s.
\end{equation}
\end{lemma}

\begin{proof}
This follows immediately from, e.g., Pham~\cite[Theorem 6.2.1]{pham book}.
\end{proof}

\begin{lemma}
\label{policy convergence estimate}
Let $F: \Omega\times[0,T]\times \mathbb R^d\times \mathbb R^d \to \mathbb R$ satisfy the hypothesis 
of Lemma~\ref{policy well defined}.
Fix $\xi \in L^2(\Omega, \mathcal F_T)$.
Let $\Phi : \mathbb{H}^2 \ni z\mapsto (Y,Z)\in  \mathcal S^2\times {\mathbb{H}}^2$, where $(Y,Z)$ is the unique solution to~\eqref{eq fixed point pia}. Moreover assume that for $z^1, z^2\in \mathbb H^2$ the following condition satisfies that there are  constants $\theta,K>0$ such that   
\begin{equation*}
|F_t(z^1_t,Z^1_t) - F_t(z^2_t,Z^2_t)| 
\leq \theta|z^1_t-z^2_t|+K|Z^1_t-Z^2_t|\,,\,\,\,t\in[0,T] \,,\,\,a.s.,
\end{equation*}
where $(Y^i,Z^i) := \Phi(z^i)$, $i=1,2$.
Then there is $\gamma > 0$ and $q\in (0,1)$ such that for any $t\in [0,T]$ we have
\begin{equation*}
\E\left[e^{\gamma t}|Y^1_t-Y^2_t|^2\right]+\|Z^1-Z^2\|_{\mathbb H^2_\gamma}^2\le q\|z^1-z^2\|^2_{\mathbb H^2_\gamma}\,.
\end{equation*}
Moreover, there is  $\gamma>0$ such that $q\in(0,1/2)$.
\end{lemma}

\begin{proof}
Consider $\gamma > 0$ which we will fix later.
We denote  $\delta z:=z^1-z^2$, $\delta Z:=Z^1-Z^2$, and $\delta Y:=Y^1-Y^2$. We then apply It\^o's formula to $e^{\gamma t}|\delta Y_t|^2$:
\begin{equation*}
\begin{split}
e^{\gamma t}|\delta Y_t|^2+&\int_t^T e^{\gamma s}|\delta Z_s|^2ds\\=& \int_{t}^{T}e^{\gamma s}\left(2\delta Y_s
(F_s(z^1_s,Z^1_s)-F_s(z^2_s,Z^2_s))-\gamma |\delta Y_s|^2\right)ds\\&-2\int_{t}^{T}e^{\gamma s} \delta Z_s \delta Y_s\, dW_s\,.
\end{split}
\end{equation*}
The expectation of the stochastic integral is $0$ due to Remark~\ref{rem stoch int unif int}.
Hence, by taking expectation we derive from the equality above that 
\begin{equation*}
\begin{split}
&\E\left[e^{\gamma t}|\delta Y_t|^2+\int_t^T e^{\gamma s}|\delta Z_s|^2ds\right] \\
&\qquad\qquad=\E\int_{t}^{T}e^{\gamma s}\left(2\delta Y_s
(F_s(z^1_s,Z^1_s)-F_s(z^2_s,Z^2_s))-\gamma |\delta Y_s|^2\right)\,ds\,.
\end{split}
\end{equation*}	
By the Lipschitz property of the generator and by the Young inequality we observe that, for any $\varepsilon>0$,
\begin{equation*}
\begin{split}
\E&\left[e^{\gamma t}|\delta Y_t|^2+\int_t^T e^{\gamma s}|\delta Z_s|^2ds\right]\le \E\int_{t}^{T}e^{\gamma s}\left(2\delta Y_s
(\theta|\delta z_s|+K|\delta Z_s|)-\gamma |\delta Y_s|^2\right)\,ds\\
&\le\E\int_{t}^{T}e^{\gamma s}\left((\theta+K)\varepsilon|\delta Y_s|^2+\theta\varepsilon^{-1}|\delta z_s|^2+K\varepsilon^{-1}|\delta Z_s|^2-\gamma |\delta Y_s|^2\right)ds\,.
\end{split}
\end{equation*}
Take $\gamma>0$ sufficiently large so that $\tilde{q}:=\max(\frac{(\theta+K)K}{\gamma},\frac{(\theta+K)\theta}{\gamma})\in(0,1/2)$. Choose $\varepsilon$ such that $\gamma=(\theta+K)\varepsilon$. Thus
\begin{equation}
\begin{split}
\label{estimate in policy iteration}
\E&\left[e^{\gamma t}|\delta Y_t|^2+\left(1-\tilde q\right)\int_t^T e^{\gamma s}|\delta Z_s|^2ds\right]\le\E\left[\int_{t}^{T}e^{\gamma s}\tilde q|\delta z_s|^2ds\right].
\end{split}
\end{equation}
Dividing by $1-\tilde q\in (1/2,1)$ we obtain
\begin{equation*}
\E\left[e^{\gamma t}|\delta Y_t|^2+\int_t^T e^{\gamma s}|\delta Z_s|^2ds\right]\le q\E\left[\int_{t}^{T}e^{\gamma s}|\delta z_s|^2ds\right],
\end{equation*}
where $q:=\frac{\tilde q}{1-\tilde q}$. Since $0<\tilde q< 1/2$ we have that $q\in (0,1)$.
Therefore from~\eqref{estimate in policy iteration} we have  for any $t\in[0,T]$ 
\begin{equation*}
\E\left[e^{\gamma t}|Y^1_t-Y^2_t|^2\right]+\|Z^1-Z^2\|_{\mathbb H^2_\gamma}^2\le q\|z^1-z^2\|^2_{\mathbb H^2_\gamma}\,.
\end{equation*}

By choosing $\gamma$ such that $\tilde{q}\in(0,1/3)$ we get that $q\in(0,1/2)$.
This concludes the proof of the lemma.	
\end{proof}

We now state a comparison principle for BSDEs.
\begin{lemma}
\label{comparison principle}
Consider the following BSDEs:
\[
Y^i_t=\xi^i+\int_t^T\phi^i(s,Z^i_s)ds-\int_t^T Z^i_s\,dW_s\,,\,\,t\in[0,T],\,\,\,i=1,2.
\]
Assume that $\xi^i\in L^2(\Omega,\mathcal{F}_T)$, $i=1,2$, and $\xi^1\le \xi^2$ a.s. Let $\phi^i:\Omega \times [0,T] \times\R^d\rightarrow\R$, $i=1,2$, be such that for all $z\in \mathbb R^d$ the processes $(\phi^i(t,z))_{t\in[0,T]}$  are progressively measurable, $\phi^i(t,0)\in\mathbb{H}^2$, and such that there is $\theta>0$ so that for all $t\in[0,T]$, $z,z'\in\R^d$ we have
\[
|\phi^i(t,z)-\phi^i(t,z')|\le\theta|z-z'|\,\,a.s.
\] 
Moreover, suppose that for $Z^1, Z^2 \in \mathbb H^2$ it holds that 
\[
\phi^1(t,Z^1_t)\le \phi^2(t,Z_t^1) \,\,\,\, \forall t \in [0,T]\,,\,\,a.s.
\]
Then $Y_t^1\le Y_t^2$ for all $0\le t\le T$ a.s.
\end{lemma}
\begin{proof}
	This follows from, e.g., Pham~\cite[Theorem 6.2.2]{pham book}.
\end{proof}

The following two lemmas are auxiliary results we need in Section~\ref{sec stability max}.
\begin{lemma}
	\label{lem grad itr pert max}
	Let $F,\bar F:\Omega \times [0,T] \times \mathbb R^d \to \mathbb R$ be measurable functions and let $F$ satisfy the hypotheses of Lemmas~\ref{bsde well defined} and~\ref{lem classic bsde iteration}.
	Fix $\xi\in L^2(\Omega,\mathcal{F}_T)$. 
	Let $\bar{z}, z, Z, \bar Z \in \mathbb H^2$ and $Y, \bar Y \in \mathcal S^2$ be 
	such that 
	\[
	\bar{Y}_t=\xi+\int_t^T\bar{F}_s(\bar{z}_s)\,ds-\int_t^T\bar{Z}_s\,dW_s\,,\,\,\, t\in [0,T]\,,
	\]
	and
	\[
	Y_t=\xi+\int_t^T F_s(z_s)\,ds-\int_t^TZ_s\,dW_s\,,\,\,\, t\in [0,T]\,.
	\]
	Then there is $\gamma>0$ and $q\in(0,1)$ such that for $t\in[0,T]$ we have
	\begin{equation*}
	e^{\gamma t}\E|\bar{Y}_t-Y_t|^2+\|\bar{Z}-Z\|^2_{\mathbb H^2_\gamma}
	\le q \|\bar{z}-z\|^2_{\mathbb H^2_\gamma}+\frac{q}{\theta}\|\bar{F}(\bar{z})-F(\bar{z})\|^2_{\mathbb H^2_\gamma}.
	\end{equation*}
\end{lemma}
\begin{proof}
	Consider $\gamma>0$ which we will fix later.
	We denote $\tilde{Y}:=\bar{Y}-Y$, $\tilde{Z}:=\bar{Z}-Z$, and $\tilde z:=\bar{z}-z$. We then apply It\^o's formula to $e^{\gamma t}|\tilde{Y}_t|^2$:
	\begin{equation*}
	\begin{split}
	e^{\gamma t}|\tilde Y_t|^2+\int_t^T e^{\gamma s}|\tilde Z_s|^2\,ds 
	& =\int_{t}^{T}e^{\gamma s}(2\tilde Y_s\, (\bar{F}_s(\bar{z}_s)-F_s(z_s))-\gamma |\tilde Y_s|^2)\,ds\\&-2\int_{t}^{T}e^{\gamma s}\tilde{Z} _s\,\tilde Y_s \, dW_s\,.
	\end{split}
	\end{equation*}
	Due to Remark~\ref{rem stoch int unif int}, the stochastic integral vanishes by taking expectation. 
	Hence
	\begin{equation}
	\begin{split}
	\label{eq estimate in the lemma for pert  via}
	\E\left[e^{\gamma t}|\tilde Y_t|^2+\int_t^T e^{\gamma s}|\tilde Z_s|^2\,ds \right] =\E\left[\int_{t}^{T}e^{\gamma s}(2\tilde Y_s\, (\bar{F}_s(\bar{z}_s)-F_s(z_s))-\gamma |\tilde Y_s|^2)\,ds\right]\,.
	\end{split}
	\end{equation}
	Notice that due to~\eqref{eq property of F in the lemma bsde well defined} for all $s\in[t,T]$ it holds that
	\[
	|\bar{F}_s(\bar{z}_s)-F_s(z_s)|\le|\bar{F}_s(\bar{z}_s)-F_s(\bar{z}_s)|+|F_s(\bar{z}_s)-F_s(z_s)|\le  |\bar{F}_s(\bar{z}_s)-F_s(\bar{z}_s)|+\theta|\bar{z}_s-z_s|.
	\]
	Then by the Young inequality we continue our estimate~\eqref{eq estimate in the lemma for pert  via}, noting that for any $\delta>0$, we have
	\begin{equation*}
	\begin{split}
	\E\Bigg[&e^{\gamma t}|\tilde Y_t|^2+ \int_t^T e^{\gamma s}|\tilde Z_s|^2\,ds\Bigg]\\
	&\le  \E\left[\int_{t}^{T}e^{\gamma s}(2|\tilde Y_s|(|\bar{F}_s(\bar{z}_s)-F_s(\bar{z}_s)|+\theta|\bar{z}_s-z_s|)-\gamma |\tilde Y_s|^2)ds\right]\\
	&\le \E\left[\int_{t}^{T}e^{\gamma s}\left((1+\theta)\delta|\tilde Y_s|^2+\delta^{-1}(\theta|\tilde z_s|^2+|\bar{F}_s(\bar{z}_s)-F_s(\bar{z}_s)|^2)-\gamma |\tilde Y_s|^2\right)ds\right].
	\end{split}
	\end{equation*}
	Fix $\gamma > (1+\theta)\theta$ and $q = (1+\theta)\theta/\gamma$. 
	Let $\delta = \gamma / (1+\theta)$.
	Then
	\begin{equation*}
	\begin{split}
	\E\left[e^{\gamma t}|\tilde Y_t|^2+\int_t^T e^{\gamma s}|\tilde Z_s|^2\,ds \right]&\le \E\left[\int_{t}^{T}e^{\gamma s} q\left(|\tilde z_s|^2+\frac1\theta|\bar{F}_s(\bar{z}_s)-F_s(\bar{z}_s)|^2\right)ds\right]\\
	&\le q \|\tilde z\|^2_{\mathbb H^2_\gamma}+\frac{q}{\theta}\|\bar{F}(\bar{z})-F(\bar{z})\|^2_{\mathbb H^2_\gamma}.
	\end{split}
	\end{equation*}
	This concludes the proof of the lemma.
\end{proof}

\begin{lemma}
	\label{lem policy impr pert max}
	Let $\bar F:\Omega \times [0,T] \times \mathbb R^d \to \mathbb R$ be a measurable function and let $F$ satisfies the hypotheses of Lemmas~\ref{policy well defined} and~\ref{policy convergence estimate}. Fix $\xi\in L^2(\Omega,\mathcal{F}_T)$. 
	Let $\bar{z},z,\bar{Z},Z\in\mathbb{H}^2$ and $\bar{Y},Y\in \mathcal S^2$ be such that
	\[
	\bar{Y}_t=\xi+\int_t^T\bar{F}_s(\bar{z}_s,\bar{Z}_s)\,ds-\int_t^T\bar{Z}_s\,dW_s\,,\,\,\, t\in [0,T]\,,
	\]
	and
	\[
	Y_t=\xi+\int_t^TF_s(z_s,Z_s)\,ds-\int_t^TZ_s\,dW_s\,,\,\,\, t\in [0,T]\,.
	\]
	Then there is $\gamma>0$ and $q\in(0,1)$ such that for $t\in[0,T]$ we have
	\begin{equation*}
	e^{\gamma t}\E|\bar{Y}_t-Y_t|^2+\|\bar{Z}-Z\|^2_{\mathbb H^2_\gamma}
	\le q \|\bar{z}-z\|^2_{\mathbb H^2_\gamma}+\frac{q}{\max(K,\theta)}\|\bar{F}_s(\bar{z}_s,\bar{Z}_s)-F_s(\bar{z}_s,\bar{Z}_s)\|^2_{\mathbb H^2_\gamma}.
	\end{equation*}
\end{lemma}
\begin{proof}
	Consider $\gamma>0$ which we will fix later.
	We denote $\tilde{Y}:=\bar{Y}-Y$, $\tilde{Z}:=\bar{Z}-Z$, and $\tilde z:=\bar{z}-z$. We then apply It\^o's formula to $e^{\gamma t}|\tilde{Y}_t|^2$:
	\begin{equation*}
	\begin{split}
	e^{\gamma t}|\tilde Y_t|^2&+\int_t^T e^{\gamma s}|\tilde Z_s|^2\,ds  \\
	&=\int_{t}^{T}e^{\gamma s}(2\tilde Y_s\, (\bar{F}_s(\bar{z}_s,\bar{Z}_s)-F_s(z_s,Z_s))-\gamma |\tilde Y_s|^2)\,ds\\&-2\int_{t}^{T}e^{\gamma s}\tilde{Z} _s\,\tilde Y_s \, dW_s\,.
	\end{split}
	\end{equation*}
	Due to Remark~\ref{rem stoch int unif int}, the stochastic integral vanishes by taking expectation. 
	Hence
	\begin{equation*}
	\begin{split}
	\E&\left[e^{\gamma t}|\tilde Y_t|^2+\int_t^T e^{\gamma s}|\tilde Z_s|^2\,ds \right] \\
	&=\E\left[\int_{t}^{T}e^{\gamma s}(2\tilde Y_s\, (\bar{F}_s(\bar{z}_s,\bar{Z}_s)-F_s(z_s,Z_s))-\gamma |\tilde Y_s|^2)\,ds\right]\,.
	\end{split}
	\end{equation*}
	Notice that by assumptions of the lemma for all $s\in[t,T]$ it holds that
	\begin{equation*}
	\begin{split}
	|\bar{F}_s&(\bar{z}_s,\bar{Z}_s)-F_s(z_s,Z_s)|\le|\bar{F}_s(\bar{z}_s,\bar{Z}_s)-F_s(\bar{z}_s,\bar{Z}_s)|+|F_s(\bar{z}_s,\bar{Z}_s)-F_s(z_s,Z_s)|\\
	&\le  |\bar{F}_s(\bar{z}_s,\bar{Z}_s)-F_s(\bar{z}_s,\bar{Z}_s)|+\theta|\bar{z}_s-z_s|+K|\bar{Z}_s-Z_s|\,.
	\end{split}
	\end{equation*}
	Then by the Young inequality for any $\delta>0$, we have
	\begin{equation*}
	\begin{split}
	\E\Bigg[&e^{\gamma t}|\tilde Y_t|^2+ \int_t^T e^{\gamma s}|\tilde Z_s|^2\,ds\Bigg]\\
	\leq &  \E\left[\int_{t}^{T}e^{\gamma s}(2|\tilde Y_s|(|\bar{F}_s(\bar{z}_s,\bar{Z}_s)-F_s(\bar{z}_s,\bar{Z}_s)|+\theta|\bar{z}_s-z_s|+K|\bar{Z}_s-Z_s|)-\gamma |\tilde Y_s|^2)ds\right]\\
	\leq & \E\Bigg[\int_{t}^{T}e^{\gamma s}\bigg((\theta+K+1)\delta|\tilde Y_s|^2+\delta^{-1}(K|\tilde{Z}|^2+\theta|\tilde z_s|^2+|\bar{F}_s(\bar{z}_s,\bar{Z}_s)-F_s(\bar{z}_s,\bar{Z}_s)|^2) \\
	& \qquad -\gamma |\tilde Y_s|^2\bigg)ds\Bigg].
	\end{split}
	\end{equation*}
	Let us take $\gamma >0$ sufficiently large so that $\tilde q := \max(\frac{(1+\theta+K)K}{\gamma}, \frac{(1+\theta+K)\theta}{\gamma}) \in (0,1/2)$.
	Let $\delta := \gamma/(1+\theta+K)$ so that
	\begin{equation*}
	\begin{split}
	\E&\left[e^{\gamma t}|\tilde Y_t|^2+\left(1-\tilde q\right)\int_t^T e^{\gamma s}|\tilde Z_s|^2\,ds \right]\\
	&\le \E\left[\int_{t}^{T}e^{\gamma s} \tilde q\left(|\tilde z_s|^2+\frac1{\max(K,\theta)}|\bar{F}_s(\bar{z}_s,\bar{Z}_s)-F_s(\bar{z}_s,\bar{Z}_s)|^2\right)ds\right]\\
	&\le \tilde q \|\tilde z\|^2_{\mathbb H^2_\gamma}+\frac{\tilde q}{\max(K,\theta)}\|\bar{F}_s(\bar{z}_s,\bar{Z}_s)-F_s(\bar{z}_s,\bar{Z}_s)\|^2_{\mathbb H^2_\gamma}.
	\end{split}
	\end{equation*}
	Dividing by $(1-\tilde q) \in (1/2, 1)$ we obtain
	\begin{equation*}
	\begin{split}
	\E & \left[e^{\gamma t}|\tilde Y_t|^2+\int_t^T e^{\gamma s}|\tilde Z_s|^2\,ds \right]
	\le q \|\tilde z\|^2_{\mathbb H^2_\gamma}+\frac{q}{\max(K,\theta)}\|\bar{F}_s(\bar{z}_s,\bar{Z}_s)-F_s(\bar{z}_s,\bar{Z}_s)\|^2_{\mathbb H^2_\gamma}\,,
	\end{split}
	\end{equation*}
	where $q := \frac{\tilde q}{1-\tilde q}$.
	Since $0 < \tilde q < 1/2$ we have that $q\in (0,1)$.
\end{proof}

\subsection{BSDE with drivers of quadratic growth.}
Since we are using BSDE theory in the proof of the main result, we would like to present some results on BSDE with drivers of quadratic growth. We refer to~\cite{Richou2012}. 

Consider the following system:
\begin{equation}\label{eq FBSDE system}
\begin{split}
X_t&=x+\int_0^tb(s,X_s)\,ds+\int_0^t\sigma(s,X_s)\,dW_s\,,\\
Y_t&=g(X_T)+\int_t^Tf(s,X_s,Z_s)\,ds-\int_t^TZ_s\,dW_s.
\end{split}
\end{equation}

\begin{theorem}[Theorem 3.6 in~\cite{Richou2012}]
	\label{thr well-defined qBSDE from Adrien's paper}
	Let $b:[0,T]\times\R^d\rightarrow\R^d$ and $\sigma:[0,T]\times\R^d\rightarrow\R^{d\times d'}$ be Lipschitz continuous with Lipschitz constant $C$ and $|b(t,0)|\le C$ and $|\sigma(t,0)|\le C$ for all $t\in[0,T]$. Let $g:\R^d\rightarrow\R$ and $f:[0,T]\times\R^d\times\R^d\rightarrow\R$ be measurable functions and let us assume that there exists constant $C$ such that for all $r\in\R^+, t\in[0,T], x, x'\in\R^d$, and $z, z'\in \R^d$
	\begin{equation*}
	\begin{split}
	&|f(t,x,z)|\le C(1+|z|^2)\,,\\
	&|g(x)|\le C\,,\\
	&|f(t,x,z)-f(t,x,z')|\le C(1+|z|+|z'|)|z-z'|\,,\\
	&|f(t,x,z)-f(t,x',z)|\le C(1+|x|^r+|x'|^r)|x-x'|\,,\\
	&|g(x)-g(x')|\le C(1+|x|^r+|x'|^r)|x-x'|\,.
	\end{split}
	\end{equation*}
	There exists a solution $(Y,Z)$ of the Markovian BSDE~\eqref{eq FBSDE system} in $\mathcal{S}^2\times\mathbb{H}^2$ and this solution is unique among solutions $(Y,Z)\in\mathcal{S}^2\times\mathbb{H}^2$ such that $Y$ is bounded. Moreover, we have
	\begin{equation*}
	|Z_t|\le C(1+|X_t|^{1+r})\,\,\forall t\in[0,T]\,,
	\end{equation*}
	and
	\begin{equation*}
	\|Z\bullet W\|_{BMO}<\infty\,,
	\end{equation*}
	where 
	\begin{equation*}
	\|M\|_{BMO}:=\sup_{\tau\in \mathcal{T}_0^T}\|\E[\langle M\rangle_T-\langle M\rangle_{\tau}\mid\mathcal{F}_{\tau}]\|_{\infty}\,;
	\end{equation*}
	here the supremum is taken over all stopping times in $[0,T]$.
\end{theorem}